\documentclass[10pt,a4paper]{amsart}
\usepackage{graphicx,color}
\usepackage{wrapfig}
\usepackage{url}

\usepackage{mathrsfs}

\newtheorem{theorem}{Theorem}[section]
\newtheorem{lemma}[theorem]{Lemma}           

\newtheorem{prop}[theorem]{Proposition}

\theoremstyle{definition}
\newtheorem{definition}[theorem]{Definition}

\theoremstyle{remark}

\numberwithin{equation}{section}

\makeatletter
\@namedef{subjclassname@2020}{%
  \textup{2020} Mathematics Subject Classification}
\makeatother


\title[Absence of EVs in continuous spectra]
{A remark on  the absence of eigenvalues in continuous spectra for discrete Schr\"{o}dinger operators on periodic lattices}
\author[K. Ando]{Kazunori Ando}
\address[K. Ando]{Graduate School of Science and Engineering,
Ehime University, Bunkyo-cho 3, Matsuyama, Ehime, 790-8577, Japan}
\email{ando.kazunori.dx@ehime-u.ac.jp}
\author[H. Isozaki]{Hiroshi Isozaki}
\address[H. Isozaki]{Graduate School of Pure and Applied Sciences, Professor Emeritus, University of Tsukuba, Tennoudai 1-1-1, Tsukuba, Ibaraki, 305-8571, Japan}
\email{isozakihiroshimath@gmail.com}
\author[H. Morioka]{Hisashi Morioka}
\address[H. Morioka]{Graduate School of Science and Engineering,
Ehime University, Bunkyo-cho 3, Matsuyama, Ehime, 790-8577, Japan}
\email{morioka.hisashi.ya@ehime-u.ac.jp}

\date{\today}

\begin{document}
\baselineskip 14pt
\maketitle

\begin{abstract} 
We prove  a Rellich-Vekua type theorem for
Schr\"{o}dinger operators with exponentially decreasing potentials on a class of lattices including square, triangular, hexagonal lattices and their ladders. 
We also discuss the unique continuation theorem and the non-existence of eigenvalues embedded in the continuous spectrum.
\end{abstract}

%%%%%%%%%%%%%%%%%%%%%%%%%%%%%%%%%%%%%%%%%%%%
%
%Section 1
%
%%%%%%%%%%%%%%%%%%%%%%%%%%%%%%%%%%%%%%%%%%%%
\section{Introduction}

The Rellich-Vekua type uniqueness theorem for the Helmholtz equation asserts an optimal decay rate at infinity for their non-trivial solutions  (\cite{Re}, \cite{Vek}).
Namely, suppose that $u\in H^2 _{loc} ({\bf R}^d)$ satisfies the equation 
\begin{equation}
-\Delta u=\lambda u \quad \quad \text{for} \quad  |x|>R_0
\label{Helmholtz1}
\end{equation}
for constants $\lambda , R_0 >0$, and the condition 
\begin{equation}
u(x)= o(|x|^{-(d-1)/2} ) \quad  \text{as} \quad |x| \to \infty .
\label{decayrateinfnity}
\end{equation}
Then there exists a constant $R_1 > R_0 $ such that $u(x)=0$ for $|x| > R_1 $. 
The decay rate (\ref{decayrateinfnity}) is sharp in the sense that there exists a non-trivial solution $u$ of (\ref{Helmholtz1})  for $x \neq 0$ which behaves like $O(r^{-(d-1)/2})$ as $r \to \infty$.  This theorem plays a crucial role in the study of continuous spectrum of elliptic self-adjoint differential operators. Restricting to the case of 2nd order operators, the Rellich-Vekua theorem implies the non-existence of eigenvalues embedded in the continuous spectrum of Schr\"{o}dinger operators (\cite{Ka}, \cite{Ro}, \cite{Ag}), guarantees the  limiting absorption principle for the resolvent (\cite{Ei}, \cite{AgHo}), and is also used in the inverse problem (see e.g. \cite{Iso20}).

A variety of techniques has been developed for the proof of Rellich-Vekua theorem  in the case of differential operators, such as  integration by parts machinery, differential inequality technique, Carleman estimates. Among them,  it is worth recalling that the first successful method used  Fourier analysis and the theory of functions of several complex variables,  reducing the issue to the problem of division in the momentum space. Namely, by passing to the Fourier transform, a linear PDE $P(D)u=f$ with constant coefficients and compactly supported $f$ is transformed to an algebraic equation $P(\xi)\widehat{u} (\xi )=\widehat{f} (\xi )$.
If $P(\xi )$ divides $\widehat{f} (\xi )$, the solution $u$ is compactly supported by virtue of  the Paley-Wiener theorem (\cite{Tre}, \cite{Lit1}, \cite{Lit2}, \cite{Ho}, \cite{Mur}). 
The absence of embedded eigenvalues for Schr\"{o}dinger operators with exponentially decaying potentials  follows from a similar argument, using the unique continuation property (\cite{RaTa}).

The above mentioned analytic method is also useful for discrete operators. In \cite{IsMo}, the Rellich-Vekua type theorem was proven for  Schr\"{o}dinger operators on the square lattice for the case where the potential is finitely supported, using  basic facts from  several complex variables and  algebraic geometry, especially Hilbert's Nullstellensatz. The idea of this proof is essentially given in the argument of \cite{ShVa}. This result was extended to super-exponentially decaying potentials by \cite{Ves} using the Paley-Wiener theorem. One should also take notice of the recent progress brought by \cite{Li}, \cite{Li2} for more general setting.

The Rellich-Vekua type uniqueness theorem is extended to locally perturbed lattice Hamiltonians  by \cite{AIM}, whose argument   highly relies on a representation of the characteristic polynomial of the discrete Laplacian on square-type lattices and triangular-type lattices.
Here the characteristic polynomial of the discrete Laplacian is a trigonometric polynomial on the flat torus ${\bf T}^d := {\bf R}^d /(2\pi {\bf Z})^d $ defined by its Fourier transform.
The result in \cite{AIM} covers a class of lattices including the square, the triangular, the hexagonal lattice and so on. For the Kagome lattice and the subdivison of the square lattice, there are examples of Hamiltonians with compactly supported potentials having embedded eigenvalues.
Applications to the inverse scattering problem for discrete Schr\"{o}dinger operators on lattices are given in \cite{IsMo2}, \cite{AIM2}, where the Rellich-Vekua type uniqueness theorem is used  to show the equivalence of the scattering matrix and the Dirichlet-to-Neumann map on a bounded domain. As we have restricted ourselves to compactly supported perturbations in \cite{AIM}, all of these forward and inverse spectral results were discussed for lattice Schr\"{o}dinger operators perturbed on a finite set.

The aim of this paper is to extend the results in \cite{AIM} for exponentially decreasing potentials. Following the idea in \cite{Ves}, we derive a growth order of eigenfunctions 
for discrete Schr\"{o}dinger operators. By the Paley-Wiener theorem, this estimate implies that super exponentially decreasing solution must vanish at infinity, which leads us to the Rellich-Vekua type theorem, Theorem \ref{Rellichtypetheorem}, the main purpose of this paper. Another barrier for the non-existence of embedded eigenvalues is  the  unique continuation property, which is  considerably different from the case of differential operators. The argument in \cite{AIM} leans heavily on the shape of the lattice. In this paper, we use the results in 
\cite{BILL1}, \cite{BILL2}, changing the geometric nature of the assumption in \cite{AIM} to the graph theoretical one. 
%This scheme covers a class of perturbed lattices broader than that of \cite{AIM}.

The plan of this paper is as follows.
In \S \ref{SecGrowthproperty}, we consider a general graph, and introduce an increasing height function to derive a growth property of solutions to the equation $(\widehat{ \Delta}_{\Gamma} + \widehat{ V} (v))\widehat{ u} (v) = 0$ (Lemma \ref{S2_lem_estimatefromabove}). It then implies that the super exponentially decaying solutions vanish identically near infinity (Lemma \ref{Lemmahatuv=0}). In \S \ref{SecLattice}, we consider a class of lattice Schr\"{o}dinger operators, and in \S \ref{SecRellichtypetheorem}, prove the Rellich type theorem. In \S \ref{ssection_ucp}, we discuss a unique continuation property and a counter example. 

\section{Growth property of eigenfunctions}
\label{SecGrowthproperty}
In this section, we consider a graph $\Gamma =  \{\mathcal V, \mathcal E\}$ with  vertex set $\mathcal V$ and edge set $\mathcal E$. 
For two vertices $v, w \in \mathcal V$, $v \sim w$ means that $v$ and $w$ are end points of an edge in $\mathcal E$. We assume that our graph is simple. Namely, there is at most one edge between any pair of two vertices and there is no loop which is an edge connecting a vertex to itself. 
For a pair of two vertices $v,w\in \mathcal V$, a path from $v$ to $w$ is a sequence $\{ v_j \} _{j=0}^N$ with a positive integer $N$ such that $v_0=v$, $v_N =w$, and $v_{j} \sim v_{j+1}$ for $j=0, 1,\ldots,N-1$.
We assume that our graph is connected, i.e.  for any pair of two vertices $v,w\in \mathcal V$, there exists a path from $v$ to $w$. For a vertex $v$,  ${\rm deg}\,v = \sharp\{w \in \mathcal V\, ; \, w \sim v \}$ is the degree of $v$.
Let $\widehat{ \Delta}_{\Gamma}$ be the Laplacian defined by
\begin{equation}
\big(\widehat{\Delta}_{\Gamma}\widehat{ u} \big)(v) = \frac{1}{\mu_v}
\sum_{w \sim v, w \in \mathcal V} g_{vw}\widehat{ u} (w),
\label{Laplacian}
\end{equation}
where $\mu_v$ and $g_{vw}$ are positive, and $g_{vw}=g_{wv} $. 
We assume that  for some constant $C > 0$
\begin{equation}
C^{-1} \leq \mu_v,  g_{vw}, {\rm deg}\, v \leq C, \quad \forall v, w \in \mathcal V.
\label{muvgvwdegvbound}
\end{equation}
For $v_0 \in \mathcal V$, we put
\begin{equation}
N_{v_0} = \{ v \in \mathcal V\, ; \, v \sim v_0\}.
\label{DefineNv0}
\end{equation}
Note that we use $\mu_v = \mathrm{deg} \, v $ and $g_{vw} =1$ in most part of this paper.

\begin{definition}
\label{Defheightfunction}
Let $\mathcal V$ be the vertex set of the graph $\Gamma = \{\mathcal V, \mathcal E\}$, and $\Omega$ a subset of $\mathcal V$. A function $h : \Omega \to {\bf Z}$ is said to be an increasing height function if it has the following two properties: \\
\noindent
(1) There exists $k_0 \geq 1$ such that 
\begin{equation}
|h(v) - h(w)| \leq k_0 \quad {\it if} \quad v, w \in \Omega, \quad v \sim w.
\label{hv-hwleqk0}
\end{equation}
(2) For any  $v \in \Omega$, there exists $v_0 \in \Omega$ such that $v \sim v_0$ and $N_{v_0} \subset \Omega$, moreover 
\begin{equation}
h(v) + 1 \leq  h(w), \quad \forall w \in N_{v_0}\setminus\{v\}.
\label{hv+1leqhw}
\end{equation}
\end{definition}

Suppose that we fix an increasing height function $h$.
For $v \in \Omega$ and
 $N_{v_0}$ in (\ref{hv+1leqhw}) with $v_0 \sim v$, we put 
\begin{equation}
D_h(v) = N_{v_0} \setminus\{v\}.
\end{equation}
%Note that $v_0$, hence $D_h(v)$,  is possibly nonunique for every $v$. 
%However, if we fix a suitable increasing height function on $\Gamma_0$, we also fix the set $D_h (v)$ so that $h(w)>h(v)$ for $w\in D_h (v)$.
Letting
\begin{equation}
D_0 = \sup_{v \in \Omega}\,\sharp 
D_h(v),
\label{DefineD0}
\end{equation}
 we have $D_0 \leq \mathrm{deg} \, v - 1$.

For $v \in \Omega$, we put
\begin{equation}
\mathbb D_{h,n}(v) = 
\left\{
\begin{split}
& D_h(v), \quad n = 1, \\
& \mathop \cup_{w'\in \mathbb D_{h,n-1}(v)}D_h(w'), \quad n \geq 2,
\end{split}
\right.
\end{equation}
\begin{equation}
\mathbb{D}_h(v) = {\mathop \cup_{n\geq 1}}\mathbb D_{h,n}(v).
\label{Dh(v)define}
\end{equation}

We regard $\mathbb D_h(v)$ from a slightly different view point. Define $D_h(v) \prec D_h(w)$  by
\begin{equation}
D_h(v) \prec D_h(w) 
\Longleftrightarrow w \in D_h(v),
\end{equation}
 and $D_h(v) \prec\prec D_h(w)$ by
\begin{equation}
D_h(v) \prec\prec D_h(w)
\Longleftrightarrow 
\left\{
\begin{array}{l}
 {\rm there} \; {\rm exist} \;  n \geq 1 \;
 {\rm and} \;  w_0, w_1, \dots,w_n \in \Omega \; {\rm such}\; {\rm that}  \\
w_0 = v, w_n = w \; {\rm and} \\
D_h(w_0) \prec D_h(w_1) \prec \cdots \prec D_h(w_n). 
\end{array}
\right.
\label{DhvprecprecDhw}
\end{equation}
Then, $\mathbb D_{h}(v)$ is the set of $w$ such that there exists a chain of $D_h(w_j)$, $1 \leq j \leq n$,  specified in (\ref{DhvprecprecDhw}). 
In view of the definition of Laplacian (\ref{Laplacian}), $\mathbb D_h(v)$ maybe called the {\it domain of  dependence} of $v$. 
%We can then consider a chain of domains of direct dependence. 

For $v, w \in \Omega$, their variation 
${\rm Var}_h(v,w)$ is defined by
\begin{equation}
{\rm Var}_h(v,w) = h(w) - h(v).
\end{equation}
Then, if $w \in D_h(v)$, we have
$1 \leq {\rm Var}_h(v,w) \leq 2k_0$. Therefore, for 
$w' \in \mathbb D_{h,n}(v)$, we have
$n \leq {\rm Var}_h(v, w') \leq 2k_0n$.  Hence
\begin{equation}
h(v) + n \leq h(w') \leq h(v) + 2k_0n, \quad w' \in \mathbb D_{h,n}(v), \quad n \geq 1.
\label{hv+nleqhwleqhv+2n}
\end{equation}

%To fix the idea,  in taking the union of the right-hand side of (\ref{Dh(v)define}), we pick up all possible choices of $D_h(w)$ for $w \in \Omega$.  
%Put for $n = 1, 2,  \dots$
%$$
%{\mathbb D}_{h,n}(v) = \{w \in {\mathbb D}_h(v)\, ; \,h(v) + n \leq h(w) \leq 
%h(v) + 2k_0n\}.
%$$

\begin{lemma} 
\label{S2_lem_estimatefromabove}
Given a subset $\Omega \subset \mathcal V$, assume that there exists an increasing height function $h : \Omega \to {\bf Z}$. 
Let  $\widehat{ V} (v)$ be a $\bf C$-valued bounded function on $\Omega$. Then, there exists a constant $C _0 > 0$ which depends only on $C$   in (\ref{muvgvwdegvbound}) and $\sup_{v \in \Omega}  |\widehat{V} (v)|$ such that 
\begin{equation}
|\widehat{ u} (v)| \leq \,  \big(C_0D_0\big)^n \sup_{w \in {\mathbb D}_{h,n}(v)}|\widehat{ u}(w)|, \quad \forall n \geq 1, \quad \forall v \in \Omega
\label{hatuvleqCnsupDhn}
\end{equation}
holds for any
$\widehat{ u}$ satisfying $(\widehat{ \Delta}_{\Gamma} + \widehat{ V}(v))\widehat{u} (v) = 0$ on $\Omega$, where $D_0$ is defined in (\ref{DefineD0}).
\end{lemma}

\begin{proof}
Take $v \in \Omega$ arbitrarily.
Using the equation $(\widehat{ \Delta}_{\Gamma} + \widehat{ V})\widehat{ u} = 0$, we have
\begin{equation}
|\widehat{ u}(v)| \leq C_0\sum_{w \in N_{v_0}\setminus\{v\}}|\widehat{ u}(w)|
\leq C_0 D_0\sup_{w \in D_{h}(v)}|\widehat{ u}(w)|, 
\label{S2_eq_hatuvleqCnsupDhn00}
\end{equation}
which proves (\ref{hatuvleqCnsupDhn}) for $n=1$. 
Assume (\ref{hatuvleqCnsupDhn}) for $n$, and take $w \in {\mathbb D}_{h,n}(v)$. 
Applying (\ref{S2_eq_hatuvleqCnsupDhn00}) (with $v, w$ replaced by $w, w'$)  to $\widehat{u}(w)$ in the right-hand side of (\ref{hatuvleqCnsupDhn}), we have
$$
|\widehat{ u}(v)| \leq \big(C_0 D_0\big)^{n+1} \, \sup_{w\in {D}_{h,n}(v), w' \in D_{h}(w)}
|\widehat{ u}(w')|.
$$
Noting that
$w \in D_{h,n}(v)$ and $w' \in D_{h}(w)$ imply
$w' \in \mathbb D_{h,n+1}(v)$,
we have proven (\ref{hatuvleqCnsupDhn}) for $n+1$. 
\end{proof}

We make a new assumption:

\medskip
\noindent
{\bf (A-1)}
{\it An infinite graph $\Gamma = \{ \mathcal{V},\mathcal{E} \} $ is given by perturbing a finite part of a periodic lattice $\Gamma_0$.
Namely, we assume that the following conditions hold true.
\begin{itemize}
\item
$\mathcal{V} = \mathcal{V} _{int} \cup \mathcal{V}_{ext} $ and $\mathcal{V}_{int} \cap \mathcal{V}_{ext} =\emptyset $ with $\sharp \mathcal{V}_{int} < \infty $.
\item
There exists an infinite, induced subgraph $\Gamma _{ext} = \{ \mathcal{V}_{ext},\mathcal{E}_{ext} \} $ of $\Gamma_0 $ without isolated vertices.
\end{itemize}
}
%{\it Except for a finite subset, $\Gamma = \{\mathcal V, \mathcal E\}$ is realized as a subset in ${\bf R}^d$. Namely, there exists a finite subset $\mathcal V_{int} \subset \mathcal V$, a constant $R > 0$ and an injection: $\mathcal V_{ext} := \mathcal V \setminus \mathcal V_{int} \to \{x \in {\bf R}^d \, ; \, |x| > R\}$ keeping the graph structure.}

\medskip
By virtue of this assumption, the induced subgraph $\Gamma_{ext} $ is realized in ${\bf R}^d$.  
For $v \in \mathcal V_{ext}$, the length of $v$, denoted by $|v|$, is defined in terms of the norm of ${\bf R}^d$. 
Thus  there is an injection $\mathcal{V}_{ext} \to \{ x\in {\bf R}^d \ ; \ |x|>R \}$ for a constant $R>0$.
We now assume that

\medskip
\noindent
{\bf (A-2)} {\it There exists a constant $a > 0$ such that for any $v \in \Omega \cap \mathcal V_{ext}$}
\begin{equation}
h(w) \leq a|w|, \quad \forall w \in \mathbb D_h(v) \cap \mathcal V_{ext}.
\end{equation}

%\medskip
%\noindent
%(C-2) {\it There exists a constant $b > 0$ such that for any $v \in \Omega$ }
%\begin{equation}
%\sharp\,  {\mathbb D}_{h,n}(v) \leq e^{bn}, \quad \forall n > 0.
%\end{equation}

\begin{lemma}
\label{Lemmahatuv=0}
Assume (A-1) and (A-2).  Let $\widehat{ u}$ satisfy $(\widehat{ \Delta}_{\Gamma} + \widehat{ V})\widehat{ u} = 0$ on $\Omega \cap \mathcal V_{ext}$. 
Take $v \in \Omega \cap \mathcal V_{ext}$ arbitrarily.  Assume that $\widehat{ u}$ is super exponentially decreasing on $\mathbb D_h(v)\cap\mathcal V_{ext}$, i.e. for any $A > 0$, there exists a constant $C_A > 0$ such that 
$|\widehat{ u}(w)| \leq C_A e^{-A|w|},  \forall w \in \mathbb D_h(v)\cap\mathcal V_{ext}$. 
Then, $\widehat{ u}(v) = 0$. 
\end{lemma}

\begin{proof}
By (A-2), for $w \in \mathbb D_{h,n}(v)$, we have $a|w| \geq h(w) \geq h(v) +n$.
Then, the right-hand side of (\ref{hatuvleqCnsupDhn}) is dominated from above by 
$$
\big(C_0D_0\big)^n C_Ae^{-A(h(v) + n)/a} \leq 
C_Ae^{-Ah(v)/a} \Big(C_0D_0e^{ - A/a}\Big)^n.
$$
We choose $A$  large enough so that $C_0D_0e^{ - A/a} < 1$. Letting $n \to \infty$, we obtain the lemma.
\end{proof}

%%%%%%%%%%%%%%%%%%%%%%%%%%%%%%%%%%%%%%%%%%%%
\section{Lattice Schr\"{o}dinger operators}
\label{SecLattice}

\subsection{Periodic lattice}
A periodic lattice is a graph $\Gamma_0 = \{\mathcal V_0,\mathcal E_0\}$ in ${\bf R}^d$, $d\geq 2$, where 
$\mathcal V_0$ is the set of vertices given by 
\begin{equation}
\mathcal V_0 = \cup_{j=1}^s \left( p_j + \mathcal L_0\right),
\label{mathcalV0=pj+L0}
\end{equation}
$\mathcal L_0$ being a ${\bf Z}$ module in ${\bf R}^d$:
\begin{equation}
\mathcal L_0=\{ {\bf v} (n) \ ; \ n\in {\bf Z}^d \} , \quad {\bf v} (n)= \sum_{j=1}^d n_j {\bf v}_j , \quad n = (n_1,\dots,n_d)
\label{DefinemathcalL0}
\end{equation}
with a basis $ {\bf v}_1, \ldots ,{\bf v}_d $ of ${\bf R}^d$, 
and $p_1 ,\ldots ,p_s$ are points in ${\bf R}^d $ satisfying 
\begin{equation}
p_j -p_k \not\in \mathcal L_0 \quad \text{if} \quad j\not= k ,
\label{Conditionpj-pkneq0}
\end{equation}
and $\mathcal E_0 \subset \mathcal V_0\times \mathcal V_0$ is the set of edges. 
By (\ref{mathcalV0=pj+L0}) and (\ref{Conditionpj-pkneq0}), there exists a bijection: $((\pi(\cdot),n(\cdot)) : \mathcal V_0 \to \{ 1,\ldots ,s\} \times {\bf Z}^d$ such that 
\begin{equation}
v=p_{\pi (v)} + {\bf v} (n(v)), \quad \forall v \in \mathcal V_0.
\label{S3_eq_pi_n}
\end{equation}
We assume that the edges  satisfy the condition
$$
(v,w)\in \mathcal E_0 \Rightarrow (m\cdot v,m\cdot w)\in \mathcal E_0, \quad 
\forall m \in {\bf Z}^d.
$$
Here $m\cdot v$ denotes the action of the group ${\bf Z}^d$ on $\mathcal V_0$ defined by
$$
{\bf Z}^d \times \mathcal V_0 \ni (m,v) \to
m\cdot v= p_{\pi (v)} +{\bf v} (n(v)+m).
$$
The Laplacian on $\Gamma_0$, denoted  by $\widehat{  \Delta}_{\Gamma_0}$, is defined by (\ref{Laplacian})
with $\mu_v = {\rm deg}\,v$, $g_{vw} = 1$. Assuming that $\Gamma_0$ is locally finite, the condition (\ref{muvgvwdegvbound}) is satisfied.

Noting that the correspondence $v \leftrightarrow (j,n) = (\pi(v), n(v))$ is bijective,  we can represent any ${\bf C}$-valued function $\hat u(v)$ on $\mathcal V_0$ as
\begin{equation}
\widehat{ u}(v) = \widehat{ u}_j(n) \quad \text{if} \quad
v = p_j + {\bf v}(n).
\end{equation}
Therefore, we represent $\widehat{ u}(v)$ as a ${\bf C}^s$-valued function $\widehat{ u}(n)= (\widehat{ u}_1(n), \dots,\widehat{ u}_s(n))^{\mathsf{T}}$ for simplicity, i.e. $\widehat{ u}_j(n)$ is defined on $p_j + \mathcal L_0$. 
Due to  the periodicity of $\Gamma_0$, $\mathrm{deg} (p_j +{\bf v} (n))$ depends only on $j$. Hence we put $\mathrm{deg}_j = \mathrm{deg} (p_j +{\bf v} (n))$.
The  Laplacian is then  written in this expression as
\begin{gather*}
(\widehat{ \Delta}_{\Gamma_0}\widehat{ u})(n)=\big( (\widehat{ \Delta}_{\Gamma_0}\widehat{ u})_1 (n),\ldots,(\widehat{ \Delta}_{\Gamma_0}\widehat{ u})_s (n)\big) ^{\mathsf{T}} , \quad n\in {\bf Z}^d , \\
(\widehat{\Delta}_{\Gamma_0}\widehat{ u})_j (n)= \frac{1}{\mathrm{deg}_j } \sum _{w\in N_{p_j  +{\bf v} (n)}}\widehat{ u}_{\pi (w)} (n(w)),
\end{gather*}
where $((\pi(\cdot),n(\cdot))$ is defined by (\ref{S3_eq_pi_n}).
A multiplication operator $\widehat{ V}(v)$ on $\ell^2(\mathcal V_0)$ is represented as  the operator of multiplication by the diagonal matrix $\widehat{V}(n)=\mathrm{diag} [\widehat{V}_1 (n),\ldots, \widehat{V}_s (n)] $ where $\widehat{V}_j (n)=\widehat{V} (p_j +{\bf v} (n)) $.
We now introduce a discrete Schr\"{o}dinger operator $\widehat{ H}=\widehat{ H}_0 +\widehat{ V}$ on $\mathcal V_0$ by
\begin{equation}
(\widehat{ H} \widehat{ u})(v)=(\widehat{ H}_0 \widehat{ u})(v)+\widehat{ V}(v)\widehat{ u}(v), \quad \widehat{ H}_0 =-\widehat{ \Delta}_{\Gamma_0}, \quad v\in \mathcal V_0.
\label{S2_eq_Schrodinger}
\end{equation}

We impose the following assumption.

\medskip
\noindent
{\bf (A-3)} \textit{  There exist constants $ \alpha , C >0$ such that 
$$
\max_{j\in \{ 1,\ldots,s \} } |\widehat{ V}_j (n)| \leq C e^{ -\alpha|n|} , \quad n\in {\bf Z}^d .
$$ 
}

\medskip

In the following, we assume that $\widehat{ V}$ is real-valued.
Then, $\widehat{ H}$ and $\widehat{ H}_0$ are bounded self-adjoint operators on the Hilbert space $\ell^2 (\mathcal V_0)= \ell^2 ({\bf Z}^d ; {\bf C}^s )$ equipped with the inner product 
$$
(\widehat{ f},\widehat{ g} ) _{\ell^2 (\mathcal V_0) } = \sum_{j=1}^s \deg_j \sum _{n \in {\bf Z}^d}  \widehat{ f}_j (n) \overline{\widehat{ g}_j (n)},
$$
and  the associated $ \| \cdot \| _{\ell^2 (\mathcal V_0)} $ norm.

\subsection{Examples} 
We introduce basic examples of lattices and Laplacians as well as 
increasing height functions.
In the following, we define $|x| = \sqrt{x_1^2 + \cdots + x_d^2} $ for $x=(x_1,\ldots,x_d) \in {\bf R}^d$. 
%\HOX{In the following, $|x|$ must be the Euclidean metric. 20240211. morioka}
%\begin{equation}
%|x| = |x_1| + \cdots + |x_d|, \quad x = (x_1,\dots,x_d) \in {\bf R}^d.
%\end{equation}
% 

\subsubsection{Square lattice} 
\label{SubsubSquarelattice}
(cf. Figure \ref{fig_squarelattice} for $d=2$)
\begin{figure}[t]
\centering
\includegraphics[width=5cm, bb=0 0 450 453]{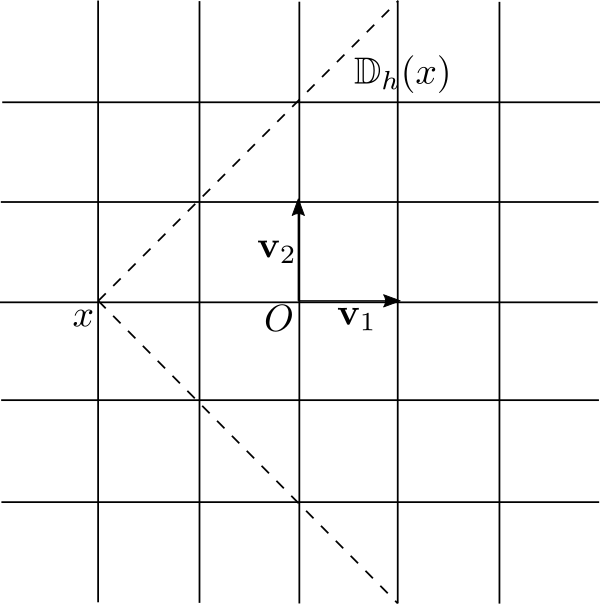}
\caption{Two-dimensional square lattice. $p_1 = 0$, ${\bf v}_1 = (1,0)$, ${\bf v}_2 = (0,1)$. $\mathbb{D}_h (x)$ is the domain of dependence for $(\widehat{H}-\lambda )\widehat{u}=0$ on the square lattice.}
\label{fig_squarelattice}
\end{figure}
Let ${\bf v}_1 = {\bf e}_1, \ldots, {\bf v}_d = {\bf e}_d$ be the standard basis of ${\bf R}^d$ and define $\mathcal L_0$ by (\ref{DefinemathcalL0}).  Take $s=1$, $p_1= 0 \in {\bf R}^d$ and let $\mathcal V_0 = {\bf Z}^d$, $\mathcal E_0= \{(v,w) \in {\bf Z}^d\times{\bf Z}^d\, ; \, |v - w| = 1\}$.
Let $N_v  = \{ w\in {\bf Z}^d \ ; \ |w-v|=1 \} $ for  $v\in \mathcal V_0$. 
The discrete Laplacian is given by 
$$
( \widehat{\Delta} _{\Gamma_0}\widehat{ u} )(n)= \frac{1}{2d} \sum _{j=1}^d \left(\widehat{  u} (n+{\bf e}_j ) + \widehat{ u} (n-{\bf e}_j ) \right) , \quad n \in {\bf Z}^d .
$$
Taking $\Omega = \mathcal V_0$ and using the Cartesian coordinates $x = (x_1,\dots,x_d)$ in ${\bf R}^d$, we define an increasing  height function $h(x)$ by 
$$
h(x) =x_1, \quad x \in \Omega.
$$
Then, in (\ref{hv-hwleqk0}), $k_0 = 1$ and for $x \in \Omega$
\begin{gather*}
D_{h}(x) = \{x + {\bf e}_1, x + 2{\bf e}_1, x+ {\bf e}_1 \pm {\bf e}_2, \ldots, x + {\bf e}_1 \pm {\bf e}_n\}, \\
\mathbb D_{h}(x)= \left\{ y\in {\bf Z}^d \ ; \ \sum _{j=2}^d | y_j - x_j | \leq y_1 -x_1 \right\}. 
\end{gather*}
Similarly, one can adopt $h(x) = x_j$, $2 \leq j \leq n$.

\subsubsection{Triangular lattice} 
\label{subsubtriangular}
(cf. Figure \ref{fig_trianglelattice} for $d=2$)
\begin{figure}[t]
\centering
\includegraphics[width=7cm, bb=0 0 638 455]{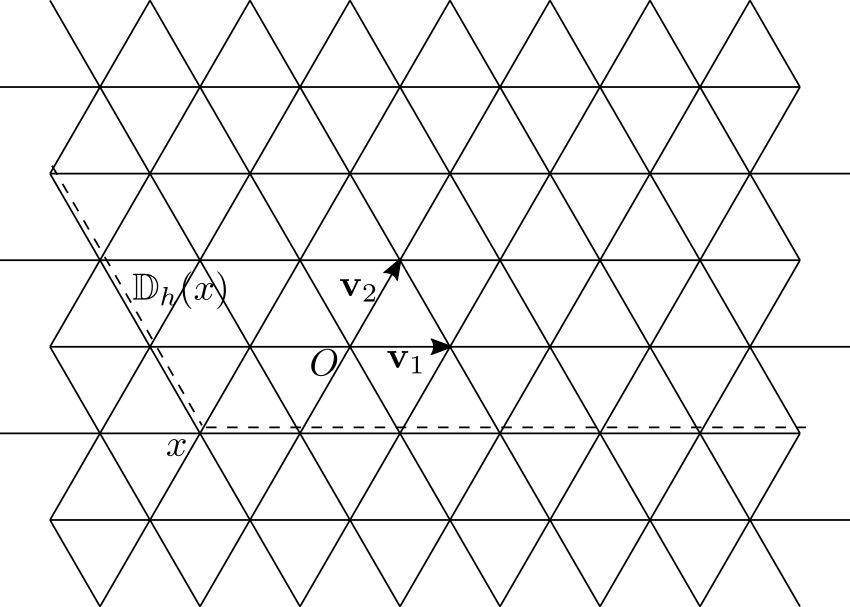}
\caption{Trianglar lattice. $p_1 = 0$, ${\bf v}_1 = (1,0)$, ${\bf v}_2 = (1/2,\sqrt{3}/2)$. $\mathbb{D}_h (x)$ is the domain of dependence for $(\widehat{H}-\lambda )\widehat{u}=0$ on the triangle lattice.}
\label{fig_trianglelattice}
\end{figure}
Let ${\bf v}_1 = (1,0)$, ${\bf v}_2 = (1/2,\sqrt{3}/2)$, $s=1$, $p_1 = 0 \in {\bf R}^2$, and  
$\mathcal V_0 = \{n_1{\bf v}_1 + n_2{\bf v}_2\, ; \, 
(n_1,n_2) \in {\bf Z}^2\}, \, \mathcal E_0 = \{(v,w) \in \mathcal V_0 \times \mathcal V_0\, ; \, |v-w| = 1\}.$ Let $N_v = \{w \in \mathcal V_0\, ; \, |v-w| = 1\}$. It is convenient to identify ${\bf R}^2$ with ${\bf C}$ and put  
$\omega = e ^{\pi i/3} = (1 + \sqrt3 i)/2$. Then $\mathcal V_0 = {\bf Z} + 
{\bf Z}\omega$.  Denoting the points of  $\mathcal V_0$ by $x = n_1 + n_2\omega$, the discrete Laplacian is defined by
\begin{equation}
\begin{split}
\big(\widehat{\Delta}_{\Gamma_0} \widehat{ u} \big)(x) &= \frac{1}{6}
\sum_{\ell=0}^5 \widehat{ u} (x + \omega^{\ell}).
\end{split}
\nonumber
\end{equation}
Taking $\Omega = \mathcal V_0$, 
 we define an increasing  height function by
$$
h(n_1 + n_2\omega) = n_1 + 2n_2.
$$
It satisfies
\begin{equation}
\left\{
\begin{split}
%& h(x) = h(x+\omega^2 ),\\ % \quad \textit{if} \quad x = x_1 + ix_2, \quad x_2 = - {x_1}/{\sqrt3}, \\
&h(x + \omega\big) = h(x) + 2, \\
&h(x + \omega^0) =h(x + \omega^2) =  h(x) + 1, \\
& h(x + \omega^4) = h(x) - 2, \\
& h(x + \omega^3) = h(x + \omega^5) = h(x) - 1.
\end{split}
\right.
\nonumber
\end{equation}
We then have $k_0 = 2$. 
For $x \in \mathcal V_0$, we take $v_0 = x_0 = x + \omega$ in (\ref{hv+1leqhw}). Then, 
$$
D_h(x) = \{x + \omega + \omega^{\ell}\, ; \, \ell =0, 1, 2, 3, 5\},
$$
$$
\mathbb D_h(x) = \{y \in \mathcal V_0\, ; \, 
y_2 - x_2 \geq - \sqrt3(y_1-x_1), \ y_2 \geq x_2
\}.
$$
By the rotation of angle $\pi/3$, we can obtain other increasing height functions.

%\begin{figure}[t]
%\centering
%\includegraphics[width=7cm]{triangle_cone.eps}
%\caption{The domain of dependence $\mathbb{D}_h (x)$ for $(H-\lambda )u=0$ on the triangle lattice. }
%\label{S2_fig_trianglecone}
%\end{figure}

\subsubsection{Hexagonal lattice.} (cf. Figure \ref{fig_hexagonallattice})
\begin{figure}[t]
\centering
\includegraphics[width=7cm, bb=0 0 375 196]{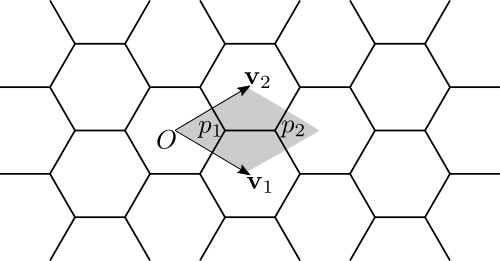}
\caption{Hexagonal lattice on ${\bf R}^2$. $p_1 = (1,0)$, $p_2 = (2,0)$. ${\bf v}_1 = (3/2,-\sqrt{3}/2)$, ${\bf v}_2 = (3/2,\sqrt{3}/2)$.}
\label{fig_hexagonallattice}
\end{figure}
\begin{figure}[t]
\centering
\includegraphics[width=7cm, bb=0 0 376 293]{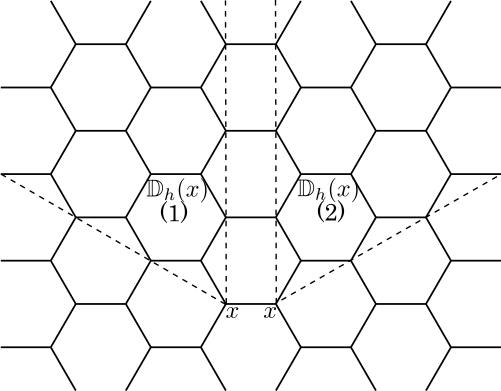}
\caption{The domain of dependence for $(\widehat{H}-\lambda )\widehat{u}=0$ on the hexagonal lattice. 
The domain $ \mathbb{D}_h (x)$ with (1) is the case where $x$ is the left-end point of a horizontal edge. The domain $ \mathbb{D}_h (x)$ with (2) is the case where $x$ is the right-end point of a horizontal edge. }
\label{S2_fig_hexcone}
\end{figure}
Let ${\bf v}_1 = (3/2,-\sqrt{3}/2)$, ${\bf v}_2 = (3/2,\sqrt{3}/2)$, $s=2$, $p_1 = (1,0)$, $p_2 = (2,0)$.
Thus $\mathcal V_0 = \mathcal V_1 \cup \mathcal V_2$ where $\mathcal V_j = p_j +\mathcal L_0$ for $j=1,2$. The edge set $\mathcal E_0$ is the set of pairs $(v,w)$ such that $v \in \mathcal V_i$, $w \in \mathcal V_j$, $i \neq j$, $|v - w| = 1$. 
For $v_i \in \mathcal V_i$, we put 
$$
N_{v_i} = \{ w\in \mathcal V_j \ ; \ |w-v_i | =1, \, j\not= i \}  .
$$
The discrete Laplacian is given by 
$$
(\widehat{ \Delta}_{\Gamma_0}\widehat{ u})(n)= \frac{1}{3} \begin{pmatrix}
 \widehat{ u}_2 (n) + \widehat{ u}_2 (n_1 -1,n_2) + \widehat{ u}_2 (n_1 , n_2 -1) \\ \widehat{ u}_1 (n ) + \widehat{ u}_1 (n_1 +1 ,n_2 ) + \widehat{ u}_1 (n_1 ,n_2 +1) 
\end{pmatrix} , \quad n\in {\bf Z}^2 .
$$

 As in \ref{subsubtriangular}, we identify ${\bf R}^2$ with ${\bf C}$ and use $\omega = (1 + \sqrt3 i)/2$. Note that  $\mathcal V_0 \subset {\bf Z} + {\bf Z}\omega$, however $ \mathcal V_0 \neq {\bf Z} + {\bf Z}\omega$. 
 In fact, $\mathcal V_0$ is the set of $p + q(1 + \omega)$, where $p, q \in {\bf Z}$ and $p \equiv 1$ or $2$ mod $3$. 
 Letting $\Omega = \mathcal V_0$, we define for $n_1 + n_2\omega \in \mathcal V_0$ 
\begin{equation}
h(n_1 + n_2\omega) = n_2.
\end{equation}
The  set of neighboring points of $x = x_1 + ix_2 \in \mathcal V_0$ is either
$$
N_x = \{x + \omega^2, x + \omega^4, x + 1\} \quad \text{for} \quad x\in \mathcal{V}_1 , 
$$
or
$$
N_x = \{x + \omega, x + \omega^3 , x + \omega^5\} \quad \text{for} \quad x\in \mathcal{V}_2 . 
$$ 
In the former case, we have
$$
h(x + \omega^4) < h(x + 1) = h(x) < h(x + \omega^2),
$$
$$
D_h(x) = \{x + \omega^2, x + \omega^2 + \omega, x + \omega^2 -1\}.
$$
and in the latter case
$$
h(x + \omega^5) <  h(x -1) = h(x) < h(x + \omega), 
$$
$$
D_h(x) = \{x + \omega, x + \omega + \omega^2, x + \omega + 1\}.
$$
Then, $h$ is an increasing height function with $k_0 = 1$. 
If  $x$ is the left-end point of a horizontal edge (see Figure \ref {S2_fig_hexcone}), we take $v_0 = x_0 = x + \omega^2$ in (\ref{hv+1leqhw}) and
$$
\mathbb D_h(x) = \{y \in \mathcal V_0\, ; \ y_1 - x_1 \leq 0, \ 
y_2 - x_2 \geq - (y_1-x_1)/\sqrt3, \  y \neq x
\}.
$$
If $x$ is the right-end point of the horizontal edge, we take 
 $x_0 = x + \omega$ and 
$$
\mathbb D_h(x) = \{y \in \mathcal V_0\, ; \, y_1 - x_1 \geq 0, 
\,  y_2 - x_2 \geq (y_1 - x_1)/\sqrt3, \ y \neq x\}.
$$

\subsubsection{Kagome lattice.} (cf. Figure \ref{fig_kagomelattice})
\begin{figure}[t]
\centering
\includegraphics[width=6cm, bb=0 0 709 548]{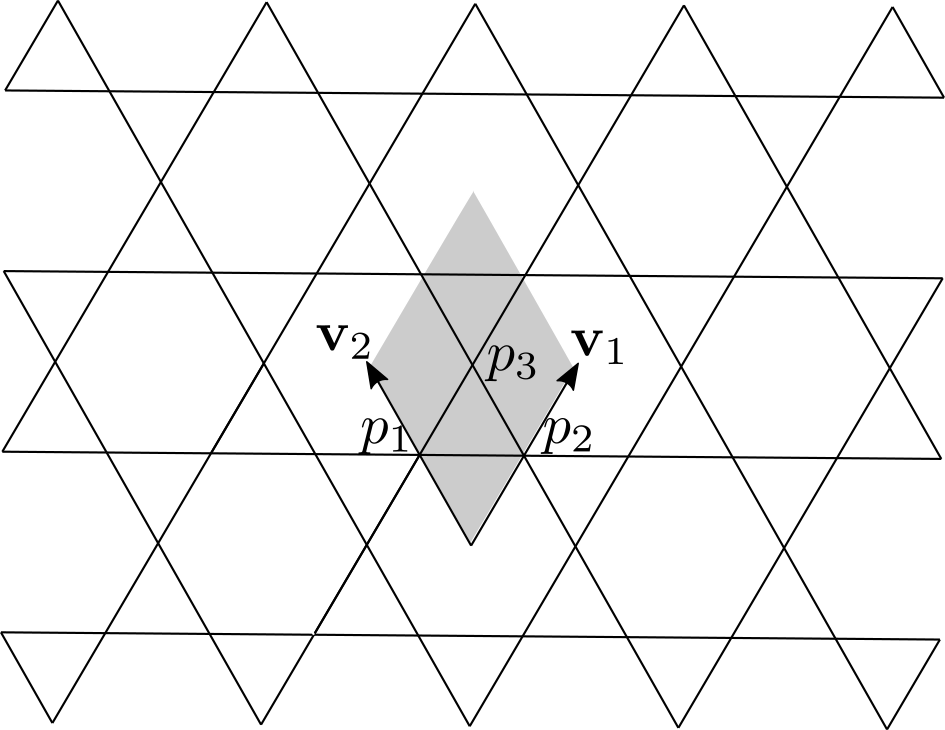}
\caption{Kagome lattice on ${\bf R}^2$. $p_1 = (0,0)$, $p_2 = (1/2,0)$, $p_3 = (1/4,\sqrt{3}/4) $. ${\bf v}_1 = (1/2,\sqrt{3} /2)$, ${\bf v}_2 = (-1/2,\sqrt{3}/2)$.
}
\label{fig_kagomelattice}
\end{figure}
Let ${\bf v}_1 = (1/2,\sqrt{3} /2)$, ${\bf v}_2 = (-1/2,\sqrt{3}/2)$, $s=3$, $p_1 = (0,0)$, $p_2 = (1/2 ,0) $, and $p_3 = (1/4,\sqrt{3}/4 )$.
We put $\mathcal V_0= \mathcal V_1 \cup \mathcal V_2 \cup \mathcal V_3 $, where $\mathcal V_j = p_j + \mathcal L_0$, $j=1,2,3 $.
For $v_j \in \mathcal V_j$, $j=1,2,3$, we put
\begin{gather*}
N_{v_j} = \{ w \in \mathcal V \ ; \ |w-v_j|=1/2 \}.
\end{gather*}
The discrete Laplacian is given by 
$$
(\widehat{ \Delta}_{\Gamma_0}\widehat{  u})(n) = \frac{1}{4} 
\begin{pmatrix}
\widehat{ u}_2 (n) + \widehat{ u}_2 (n_1 -1,n_2 +1) + \widehat{ u}_3 (n )+ \widehat{ u}_3 (n_1 -1 ,n_2 ) \\
\widehat{ u}_1 (n ) + \widehat{ u}_1 (n_1 +1,n_2 -1) + \widehat{ u}_3 (n) + \widehat{ u}_3 (n_1 ,n_2 -1) \\
\widehat{ u}_1 (n ) + \widehat{ u}_1 (n_1 +1 ,n_2 ) +\widehat{  u}_2 (n )+\widehat{ u}_2 (n_1 ,n_2 +1)
\end{pmatrix}
$$
for $n \in {\bf Z}^2$.
We cannot construct increasing height functions for the Kagome lattice. This problem is closely related with the existence of compactly supported eigenvectors, which we will discuss in \ref{SubsubSCounterexample}.

\subsubsection{Ladder.} (cf. Figure \ref{fig_ladder} for the case $d=2$)
\begin{figure}[t]
\centering
\includegraphics[width=6cm, bb=0 0 323 315]{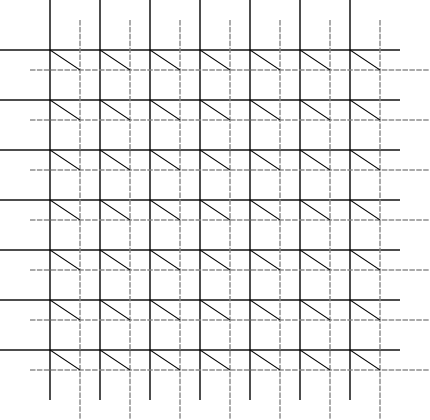}
\caption{Ladder of the square lattice ${\bf Z}^2$ in ${\bf R}^3 $. $p_1 = (0,0,0)$, $p_2 = (0,0,1)$. ${\bf v}_1 = (1,0,0)$, ${\bf v}_2 = (0,1,0)$.}
\label{fig_ladder}
\end{figure}
The ladder of the square lattice ${\bf Z}^d$ is realized in ${\bf R}^{d+1}$.
Let $\mathcal L_0= \{ (n_1,\ldots,n_d,0) \ ; \ n_1,\ldots ,n_d \in {\bf Z} \} $.
Letting $p_1 = (0,\ldots,0,0)$ and $p_2 =(0,\ldots,0,1) $, we put $\mathcal V=\mathcal V_1 \cup \mathcal V_2$ where $\mathcal V_j = p_j + \mathcal L_0$, $j=1,2 $.
For $v_j \in \mathcal V_j $ we put
$$
N_{v_j} = \{ w\in \mathcal V \ ; \ |w-v_j | =1 \} , \quad j=1,2.
$$
The discrete Laplacian on the ladder is given by 
$$
(\widehat{ \Delta}_{\Gamma_0}\widehat{ u} )(n)= \frac{1}{2d+1}
\begin{pmatrix}
\widehat{ u}_2 (n)+ \sum_{j=1}^d (\widehat{ u}_1 (n+ {\bf e}_j) + \widehat{ u}_1 (n-{\bf e}_d)) \\ \widehat{ u}_1 (n)+ \sum_{j=1}^d (\widehat{ u}_2 (n+ {\bf e}_j) +\widehat{ u}_2 (n-{\bf e}_d))
 \end{pmatrix}
 , \quad n \in {\bf Z}^d .
$$

We consider the case $d=2$. 
As the increasing height function, we define  
\begin{equation}
h(v) = x_{1} \quad \textit{for} \ \ x = (x_1,x_2,x_{3}) \in \mathcal V_0.
\end{equation}
Then,  $k_0 = 1$. We have for $x = (x_1,x_2,0)$, 
$$
\mathbb D_h(x) = \{(y_1,y_2,0), (y_1',y_2',1) \in \mathcal V_0\, ; \, 
y_1 - x_1 \geq |y_2 - x_2|, y_1' - (x_1+1) \geq |y_2'-x_2|
\},
$$
and for $x = (x_1,x_2,1)$,
$$
\mathbb D_h(x) = \{(y_1,y_2,0), (y_1',y_2',1) \in \mathcal V_0\, ; \, 
y_1 - (x_1+1) \geq |y_2 - x_2|, y_1' - x_1 \geq |y_2'-x_2|
\}.
$$

%%%%%%%%%%%%%%%%%%%%%%%
\subsection{Momentum representation and Fermi surface}
Let ${\bf T}^d = {\bf R}^d / (2\pi {\bf Z})^d$ be the flat torus.
The Fourier transformation $\mathcal{U}_{\Gamma_0}$ on $\Gamma_0$ is defined by 
$$
{f}_j (x) := (\mathcal{U}_{\Gamma_0} \hat f )_j (x)= (2\pi )^{-d/2} \sqrt{\mathrm{deg}_j} \sum_{n \in {\bf Z}^d} e^{in\cdot x} \widehat{ f}_j (n), \quad x \in {\bf T}^d ,
$$
for $\widehat{ f}(n)=[\widehat{ f}_1 (n), \ldots,\widehat{ f}_s (n)]^{\mathsf{T}}$, $n \in {\bf Z}^d $.
Letting $L^2 ({\bf T}^d ; {\bf C}^s )$ be equipped with the inner product 
$$
(f,g ) _{L^2 ({\bf T}^d ; {\bf C}^s )} = \sum_{j=1}^s \int _{{\bf T}^d} {f}_j (x) \overline{{g}_j (x)} dx,
$$
$ \mathcal{U}_{\Gamma_0} : \ell^2 (\mathcal V_0) \to L^2 ({\bf T}^d ; {\bf C}^s )$ is a unitary operator.
For the standard basis $ {\bf e}_1 , \ldots , {\bf e}_d $ of ${\bf R}^d$, and the conjugation by $\mathcal{U}_{\Gamma_0}$, the shift $\widehat{ u} (\cdot) \to \widehat{ u}(\cdot +{\bf e}_k )$ on ${\bf Z}^d$ is transformed to the multiplication by $e^{-i x_k} $ on ${\bf T}^d$.
Then we have 
\begin{equation}
\widehat{H}_0  = - \widehat{\Delta}_{\Gamma_0} = \mathcal{U}_{\Gamma_0}^{-1}H_0 \, \mathcal{U}_{\Gamma_0},
\label{S2_eq_Laplaciantorus}
\end{equation}
where $H_0$ is the operator of multiplication by an $s\times s$ Hermitian matrix $H_0 (x)$ whose entries are trigonometric polynomials (see also Section 2.2 in \cite{AIM}).

Let ${\bf T}^d_{{\bf C}} = {\bf C}^d / (2\pi {\bf Z})^d $ be  the complex flat torus. We put
\begin{gather}
p(x,\lambda )=\det ({H}_0 (x)-\lambda ) , \quad \lambda \in {\bf C} , \label{S2_eq_chpoly} \\
M_{\lambda}  = \{x \in {\bf T}^d \ ; \ p(x,\lambda )=0 \} , \label{S2_eq_fermi} \\
M _{\lambda} ^{\bf C} = \{z \in {\bf T}^d _{{\bf C}} \ ; \ p(z,\lambda )=0 \} , \label{S2_eq_complexfermi} \\
M_{\lambda,reg}^{\bf C} = \{z \in M_{\lambda}^{\bf C} \ ; \ \nabla_{z} \, p (z,\lambda ) \not= 0 \} , \quad \label{S2_eq_regfermi} \\
M_{\lambda,sng}^{\bf C} = \{z \in M_{\lambda}^{\bf C} \ ; \ \nabla_{z}\,  p (z,\lambda ) = 0 \} , \label{S2_eq_singfermi} \\
\widetilde{\mathcal{T}} = \{ \lambda \in \sigma (H_0 ) \ ; \ M_{\lambda,sng}^{\bf C} \cap {\bf T}^d \not= \emptyset \} . \label{S2_eq_threshholds}
\end{gather}
In order to prove the Rellich type uniqueness theorem on $\Gamma_0$, we impose the following assumption on the torus structure of $M_{\lambda}^{\bf C}$.

\medskip
\noindent
{\bf (A-4)} 
\textit{There exists a subset $ \mathcal{T}_1 \subset \sigma (H_0)$ such that for $\lambda \in \sigma (H_0)\setminus \mathcal{T}_1 $ 
\begin{enumerate}
\item $M_{\lambda,sng}^{{\bf C}}$ is at most a discrete set.
\item Let $\Omega_a = \{ z \in {\bf T}^d_{{\bf C}} \ ; \ |\mathrm{Im} \, z|<a \} $.
For an arbitrarily fixed $a>0$, each connected component of $M_{\lambda,reg}^{{\bf C}} \cap \Omega_a $ intersects with ${\bf T}^d$ and the intersection is a $(d-1)$-dimensional real analytic submanifold of ${\bf T}^d$.
\end{enumerate}
}

\medskip
%We say that the  lattice $(L,X,E)$ is \textit{admissible} if (A-4) is satisfied.
Let us check the assertion (1) in the assumption (A-4) for the examples given above.
For the proof, see \cite[Section 3]{AIM}.

\medskip
\noindent
\textit{Square lattice.}
\begin{itemize}
\item $
p(x, \lambda )= - \dfrac{1}{d} \sum _{j=1}^d \cos x_j -\lambda.
$
\item $ \sigma (H_0 )= [-1,1] $. 
\item $\widetilde{\mathcal{T}} = \{ n/d \ ; \ n=-d,-d+2,\ldots,d-2,d \} $.
\item $ \mathcal{T}_1 = \{ -1,1\} $.
\end{itemize}

\medskip
\noindent
\textit{Hexagonal lattice.}
\begin{itemize}
\item $H_0 (x)= -\dfrac{1}{3} \begin{pmatrix} 0 & 1+e^{ix_1} + e^{ix_2} \\  1+e^{-ix_1} + e^{-ix_2} & 0 \end{pmatrix}.$
\item $p(x,\lambda )= \lambda^2 - \dfrac{1}{9} \left( 3+2( \cos x_1 +\cos x_2 + \cos (x_1 - x_2 )) \right)$.
\smallskip
\item $ \sigma (H_0 )= [-1,1] $. 
\item $ \widetilde{\mathcal{T}} = \{ -1,-1/3, 0, 1/3, 1 \} $.
\item $ \mathcal{T}_1 = \{ -1,0,1\} $.
\end{itemize}

\medskip
\noindent
\textit{Kagome lattice.}
\begin{itemize}
\item $
H_0 (x)= -\dfrac{1}{4} \begin{pmatrix} 0 & 1+e^{i (x_1 - x_2 )} & 1+e^{ix_1} \\ 1 +e^{-i(x_1 -x_2 )} & 0 & 1+e^{ix_2}  \\ 1+e^{-ix_1} & 1+e^{-ix_2} & 0    \end{pmatrix} ,
$
\item 
$ p(x,\lambda )=- \left( \lambda - \frac{1}{2}
 \right) \left( \lambda^2 + \frac{\lambda}{2} - \frac{\beta (x) }{8 } \right)
$, \ 
$
\beta (x)= 1+\cos x_1 +\cos x_2 + \cos (x_1 - x_2 )$, 
\item $\sigma (H_0 )=[-1,1/2] $.
\item $ \widetilde{\mathcal{T}} = \{ -1, -1/2, -1/4, 0, 1/2 \}$.
\item $ \mathcal{T}_1 = \{ -1,-1/4,1/2\} $.

\end{itemize}

\medskip
\noindent
\textit{Ladder.}
\begin{itemize}
\item
$
H_0 (x)= -\dfrac{1}{2d+1} \begin{pmatrix} 2 \sum _{j=1}^d \cos x_j & 1 \\ 1 & 2\sum_{j=1}^d \cos x_j \end{pmatrix} ,
$
\item 
$
p(x, \lambda )= p_+ (x, \lambda ) p_- (x,\lambda ) , 
$ \
$
p_{\pm} (x, \lambda )= \lambda + \dfrac{1}{2d+1} \left( 2\sum_{j=1}^d \cos x_j \pm 1  \right),
$
\item $ \sigma ( H_0 )= [-1,1] $.
\item $\widetilde{\mathcal T} = \Big\{\dfrac{-2d+1}{2d+1}, \dfrac{-2d+5}{2d+1},\dots,1
\Big\}\cup\Big\{-1, \dfrac{-2d+3}{2d+1},\dots,\dfrac{2d-1}{2d+1}
\Big\}$
\item $
\mathcal{T}_1 = [-1,(-2d+1)/(2d+1)] \cup [(2d-1)/(2d+1),1] .
$
\item For $   \lambda <(-2d+1)/(2d+1)$, we have $p_+ (\xi ,\lambda ) \not=0 $ for any $\xi \in {\bf T}^d $.

\item  For $ \lambda > (2d-1)/(2d+1)$, we have $p_- (\xi ,\lambda ) \not=0 $ for any $\xi \in {\bf T}^d $.

\end{itemize}

It is not an easy task to prove the assertion (2) in (A-4) for specific lattices. 
In the Appendix, we give a proof for the square lattice and the hexagonal lattice.

\medskip

\textit{Remark.}
The assumption (A-4) is used in the proof of Theorem \ref{S3_thm_Rellich} as well as the Paley-Wiener theorem.
Suppose that we replace the assumption (A-3) by 

\medskip

\noindent
{\bf (A-3)'} \textit{There exists a constant $C>0$ such that 
$$
\max _{j\in \{ 1,\ldots ,s\} } |\widehat{V} (n)| \leq Ce^{ -\alpha |n| }, \quad n\in {\bf Z}^d,
$$
for any $ \alpha >0 $.}

\medskip

\noindent
The simplest case of (A-3)' is the case of $\sharp \mathrm{supp} \, \widehat{V} <\infty $.
Under the assumption (A-3)', the assumption (A-4) can be simplified as 

\medskip

\noindent
{\bf (A-4)'}  \textit{
There exists a subset $ \mathcal{T}_1 \subset \sigma (H_0)$ such that for $\lambda \in \sigma (H_0)\setminus \mathcal{T}_1 $ 
\begin{enumerate}
\item $M_{\lambda,sng}^{{\bf C}}$ is at most a discrete set.
\item Each connected component of $M_{\lambda,reg}^{{\bf C}} $ intersects with ${\bf T}^d$ and the intersection is a $(d-1)$-dimensional real analytic submanifold of ${\bf T}^d$.
\end{enumerate}
}

\medskip

\noindent
The assumption (A-4)' was used in \cite{AIM,AIM2} in which the authors assumed that $\hat{V} (n)= 0$ except for a finite number of $n\in {\bf Z}^d$.
In \cite{Ves}, the Rellich-Vekua type uniqueness theorem was generalized to super-exponentially decreasing potentials.

\section{Rellich type uniqueness theorem}
\label{SecRellichtypetheorem}
\subsection{Analyticity of eigenfunction in momentum space}

The purpose of this section is to prove the following theorem.
\begin{theorem}
\label{Rellichtypetheorem}
Assume (A-1)-(A-4).
Let $ \lambda \in \sigma _{ess} (H ) \setminus \mathcal{T}_1 $.
Suppose that a function $\widehat{ u}$ on $\mathcal{V} _{ext}$ satisfies $\widehat{ H} \widehat{ u} =\lambda \widehat{ u}$ for $|n| > R_0$ with some constant $R_0$ and
\begin{equation}
\lim _{R\to \infty} \frac{1}{R} \sum_{j=1}^s \sum _{|n| < R} | \widehat{ u}_j (n)|^2 =0.
\label{S3_eq_B*0}
\end{equation}
Then there exists a  constant $R_1 > R_0$ such that $\widehat{ u} (n)=0$ for $|n| >R_1$. 
\label{S3_thm_Rellich}
\end{theorem}

To prove this theorem, we first rewrite the equation $\widehat{H} \widehat{ u}=\lambda \widehat{ u}$ into the form
\begin{equation}
(\widehat{ H}_0 -\lambda )\widehat{ u}=\widehat{ f}, \quad \widehat{ f}=-\widehat{ V}\widehat{ u} .
\label{S3_eq_helmholtz}
\end{equation}
Passing to the Fourier transformation, we have the equation on the torus
$$
(H_0(x) - \lambda ) u(x) = f(x), \quad x \in {\bf T}^d.
$$
Multiplying this equation by the cofactor matrix of ${H}_0 (x)-\lambda $, we have 
\begin{equation}
p(x,\lambda ) u(x) = g(x), \quad x\in {\bf T}^d ,
\label{S3_eq_helmholtz_single}
\end{equation}
for a distribution ${g}$ on ${\bf T}^d$.
Picking up one of the components of $u$ and $g$, and denoting them by $u$ and $g$ again, we consider the equation $p(x,\lambda )u=g$.
In view of (\ref{S3_eq_B*0}) and (A-3), we have 
$$
\max_{j\in \{ 1,\ldots ,s\}} |\widehat{ f}_j (n)| \leq C e^{-\alpha|n|} , \quad n\in {\bf Z}^d ,
$$
 for  constants $C, \alpha >0$.
Note that the components of the cofactor matrix of ${H}_0 (x)-\lambda $ are trigonometric polynomials. 
As in \cite[Theorem 6.1]{Ves}, we use the following Paley-Wiener theorem.

\begin{theorem}
Let $\widehat{ q} \in \ell^2 ({\bf Z}^d ; {\bf C})$ and $ \gamma_0 >0$.
Then $e^{\gamma \langle \cdot \rangle } \widehat{ q} \in \ell^2 ({\bf Z}^d ; {\bf C})$ for all $\gamma \in (0,\gamma_0 )$ if and only if the function $q$ extends to an analytic function in $ \{ \zeta \in {\bf T}^d _{{\bf C}} \ ; \ | \mathrm{Im} \, \zeta |< \gamma_0  \} $.
\label{S3_thm_PW}
\end{theorem}

As a consequence, we obtain the analyticity of the function $g$ in (\ref{S3_eq_helmholtz_single}). 
\begin{lemma}
The function $g$ on the right-hand side of (\ref{S3_eq_helmholtz_single}) extends to an analytic function in the subset $\{ \zeta \in {\bf T}^d _{{\bf C}} \ ; \ | \mathrm{Im} \, \zeta | < \alpha  \} $.
\label{S3_lem_analytic_g}
\end{lemma}

We also have the regularity of the solution $u$.
By virtue of the analyticity of ${g} $, the following lemma is proven in the same way as in the proof of \cite[Lemma 4.3]{IsMo} or \cite[Lemma 5.2]{AIM}.
For the sake of the completeness of our argument, we reproduce an outline of the proof.

\begin{lemma}
Let $\lambda$ be as in Theorem \ref{S3_thm_Rellich}.
Then $u \in C^{\infty} ({\bf T}^d \setminus M _{\lambda, sng}^{{\bf C}} )$.
In particular, we have ${g} (x)=0 $ on $M _{\lambda, reg}^{{\bf C}} \cap {\bf T}^d $.
\label{S3_lem_regularity_u}
\end{lemma}

Proof. 
Suppose that $u$ and $g$ satisfy the equation (\ref{S3_eq_helmholtz_single}).
We take an arbitrary point $x^{(0)} \in M_{\lambda ,reg}^{{\bf C}} \cap {\bf T}^d $.
Let $U$ be a sufficiently small neighborhood of $x^{(0)} $ in ${\bf T}^d $ such that $ M_{\lambda,sng}^{{\bf C}} \cap U = \emptyset $.
We choose $\chi \in C^{\infty} ({\bf T}^d )$ satisfying $\mathrm{supp} \, \chi \in U$ and $\chi (x^{(0)} )=1 $.
Since we have $\nabla_x  p(x^{(0)},\lambda )\not= 0$, we can take a local coordinates $x\mapsto y=(y_1 ,y')$ for $x\in U$ such that $y_1 = p(x,\lambda ) $.
Let $u_{\chi} =\chi u$ and $g_{\chi} =\chi g $.
Thus we have $y_1 u_{\chi} = g_{\chi} $. 
We denote by $\widetilde{u} _{\chi} $ and $\widetilde{g} _{\chi} $ the Fourier transforms of $u _{\chi} $ and $g_{\chi} $.

As has been given in \cite[Section 4]{AIM}, the condition (\ref{S3_eq_B*0}) is equivalent to 
$$
\lim _{R\to \infty} \frac{1}{R} \sum _{j=1}^s \| {\bf 1} (\sqrt{-\Delta_x} <R)u_j \|^2 _{L^2 ({\bf T}^d )} =0,
$$
where ${\bf 1} (\lambda <a)$ is the characteristic function for the subset $(-\infty,a)$.
It follows from the above equality and (\ref{S3_eq_helmholtz_single}) that 
\begin{equation}
\lim _{R\to \infty} \frac{1}{R} \int _{|\eta|<R} |\widetilde{u}_{\chi} (\eta)|^2 d\eta =0.
\label{S4_eq_Lemproof001}
\end{equation}

The equation $y_1 u_{\chi} = g_{\chi} $ implies 
$$
\frac{\partial \widetilde{u}_{\chi}}{\partial \eta_1} =- i\widetilde{g} _{\chi} .
$$
Integrating this equation, we have
$$
\widetilde{u} _{\chi} (\eta)=- i\int_0^{\eta_1} \widetilde{g}_{\chi} (s,\eta')ds + \widetilde{u}_{\chi} (0,\eta').
$$
Note that $\widetilde{g}_{\chi} (s,\eta') $ is rapidly decreasing as $|s|\to \infty$ since $g_{\chi} \in C^{\infty} (U)$ with a small support.
Then we can see that the limit
\begin{equation}
\lim _{\eta_1 \to \infty} \widetilde{u} _{\chi} (\eta)=- i\int_0^{\infty} \widetilde{g}_{\chi} (s,\eta')ds + \widetilde{u}_{\chi} (0,\eta')
\label{S4_eq_Lemproof002}
\end{equation}
exists.
Actually this limit vanishes due to (\ref{S4_eq_Lemproof001}).
In order to see this, we take a subset 
$$
D_R = \left\{ \eta \in {\bf R}^d \ ; \ |\eta'|<\delta R, \quad \frac{R}{3} <\eta_1 <\frac{2R}{3} \right\} , \quad \delta \in (0,1).
$$
For a sufficiently small $\delta >0$, we see $D_R \subset \{ \eta \in {\bf R}^d \ ; \ |\eta|<R \} $.
Then we obtain as $R\to \infty$
\begin{gather*}
\begin{split}
\frac{1}{R} \int _{D_R} |\widetilde{u}_{\chi} (\eta)|^2 d\eta &= \frac{1}{R} \int _{|\eta'|<\delta R} \int _{R/3}^{2R/3} |\widetilde{u}_{\chi} (\eta_1,\eta') |^2 d\eta_1 d\eta' \\
&\leq \frac{1}{R} \int _{|\eta|<R} |\widetilde{u}_{\chi} (\eta )|^2 d\eta \to 0 ,
\end{split}
\end{gather*}
due to (\ref{S4_eq_Lemproof001}).
This estimate implies $\liminf_{\eta_1 \to \infty} |\widetilde{u}_{\chi} (\eta_1 ,\eta')| =0$.
Then (\ref{S4_eq_Lemproof002}) shows 
$$
\widetilde{u} _{\chi} (\eta )= i\int _{\eta_1} ^{\infty} \widetilde{g} _{\chi} (s,\eta')ds ,
$$
and we see that $\widetilde{u} _{\chi} (\eta)$ is rapidly decreasing when $\eta_1 \to \infty $.
We also see that $\widetilde{u} _{\chi} (\eta)$ is rapidly decreasing when $\eta_1 \to -\infty $ in a similar way.
Obviously, $\widetilde{u} _{\chi} $ decays rapidly as $|\eta'|\to \infty$.
Hence $u_{\chi} \in C^{\infty} (U)$.
It is trivial to show that $u$ is smooth outside $M_{\lambda} $.
Now we have $u \in C^{\infty} ({\bf T}^d \setminus M _{\lambda, sng}^{{\bf C}} )$.
The equation (\ref{S3_eq_helmholtz_single}) implies $g(x)=0$ for all $x\in M_{\lambda ,reg}^{{\bf C}} \cap {\bf T}^d $.
\qed

\medskip

\textit{Proof of Theorem \ref{S3_thm_Rellich}.}
To begin with, we show that $\hat u$ is super-exponentially decreasing at infinity on ${\bf Z}^d$.
In a neighborhood of an arbitrary point in $M_{\lambda ,reg}^{{\bf C}} \cap {\bf T}^d$, we take  local coordinates $(\zeta_1 , \ldots ,\zeta_d ) $ so that $M_{\lambda , reg}^{{\bf C}} $ is represented as the hyper-surface $\{ \zeta \in {\bf C}^d \ ; \  \zeta_d =0 \}$.
Let $ \zeta = \xi +i\eta $ for $\xi , \eta \in {\bf R}^d $. 
We expand ${g} $ into a Taylor series on $ M_{\lambda ,reg}^{{\bf C}} $ as 
$$
{g} | _{M_{\lambda ,reg}^{{\bf C}}} = \sum_{\nu' =(\nu_1 , \ldots , \nu_{d-1})} c_{\nu'} \zeta_1 ^{\nu_1} \cdots \zeta_{d-1}^{\nu_{d-1}} .
$$
By Lemma \ref{S3_lem_regularity_u}, ${g} | _{M_{\lambda,reg}^{{\bf C}}} $ vanishes for $ \eta_1 = \cdots  = \eta_{d-1} =0$.
Then we have $ c_{\nu' } =0$ and hence ${g} (\zeta )=0$ in a neighborhood of $M_{\lambda ,reg}^{{\bf C}} \cap {\bf T}^d $. 
Recalling the assertion 2 in the assumption (A-4), by the analytic continuation, we see that ${g} (\zeta )$ vanishes in $ M_{\lambda ,reg} ^{{\bf C}} \cap \{ \zeta \in {\bf T}^d _{{\bf C}} \ ; \ | \mathrm{Im} \, \zeta | < \alpha \} $.
Thus, the singularities of ${g} (\zeta ) / p(\zeta,\lambda )$ are in $ M_{\lambda ,sng}^{{\bf C}} $.
Then the set of these singularities consists of at most discrete subset in ${\bf T}^d_{{\bf C}} $ with $d\geq 2$ in view of the assertion 1 in the assumption (A-4).
It is well-known that by Hartogs' extension theorem (see e.g. \cite[Corollary 7.3.2]{Kr}) discrete singularities are removable. 
Then ${u} (\zeta )= {g}(\zeta )/p(\zeta ,\lambda )$ is analytic in $\{ \zeta \in {\bf T}^d_{{\bf C}} \ ; \ | \mathrm{Im} \, \zeta | < \alpha \} $. 
This implies $ e^{\gamma \langle \cdot \rangle}\widehat{ u} \in \ell^2 ({\bf Z}^d ; {\bf C} )$ for any $\gamma \in ( 0,\alpha )$ by virtue of Theorem \ref{S3_thm_PW}.

Due to (\ref{S3_eq_helmholtz}) and (\ref{S3_eq_helmholtz_single}), we have $e^{2\gamma \langle \cdot \rangle} \widehat{g} \in \ell^2 ({\bf Z}^d ; {\bf C} )$ for any $ \gamma \in (0, \alpha )$. 
By the above argument, $\hat{u}$ gains the decay of order $\alpha $ i.e. $ e^{2\gamma \langle \cdot \rangle} \widehat{u} \in \ell^2 ({\bf Z}^d ; {\bf C} ) $ for any $\gamma \in (0, \alpha )$.
Repeating this procedure,  we have $ e^{N\gamma \langle \cdot \rangle}\widehat{ u} \in \ell^2 ({\bf Z}^d  ;  {\bf C} )$ for any positive integers $N$ and any $ \gamma \in (0,\alpha )$.
We can then apply Lemma \ref{Lemmahatuv=0} to conclude 
Theorem \ref{S3_thm_Rellich}. 
\qed

%%%%%%%%%%%%%%%%%%%
\section{Unique continuation property} 
\label{ssection_ucp}
\subsection{UCP}
We use the terminology  unique continuation property (UCP for short) for the Schr\"{o}dinger operator $\widehat{ H}$ on a periodic lattice $\Gamma_0 = \{ \mathcal{V}_0 , \mathcal{E}_0 \}$ in the following sense.
\begin{definition}
We say that 
the lattice $\Gamma_0$ satisfies UCP if the following property holds true:
Suppose that a solution $\widehat{ u}$  to the equation $ (\widehat{ H} -\lambda )\widehat{ u} =0$, $\lambda$ being a constant, on $\mathcal V_0$ satisfies $\widehat{ u} (n)=0$ for $|n| >R $ for some $R > 0$.
Then  $\hat u$ vanishes identically on $\mathcal V_0$.  
\label{S2_def_ucp}
\end{definition}
For elliptic PDEs, UCP is often used in the following sense:
Let $P(D)$ be a  differential operator, and consider the PDE $P(D)u=0$ in a domain $\Omega \subset {\bf R}^d$. 
If $u=0$ in an open subset $\Omega' \subset \Omega $, then $u=0$ identically in $\Omega$. 

For the discrete case, UCP in this usual sense does not hold true.
Even though we restrict to the case as in Definition \ref{S2_def_ucp}, the UCP  may fail for some lattices or potentials.
In fact, for the  kagome lattice,  if we choose a suitable potential $\widehat{ V}$, we can construct an eigenfunction of $\widehat{ H}$ whose support consists of a finite number of vertices (see \cite[Section 5.3]{AIM}).

A sufficient condition for UCP was derived in \cite{BILL1,BILL2}. 
Let us recall their arguments.
The key role is played by \textit{extreme points} and \textit{two-points condition} for finite graphs with boundary.
Suppose that $ \widetilde{\Gamma } = \{ \widetilde{\mathcal{V}}, \widetilde{\mathcal{E}} \}$ is a finite graph i.e. $\sharp \widetilde{\mathcal{V}} < \infty $ and $\sharp \widetilde{\mathcal{E}} < \infty $.
We decompose the set of vertices $\widetilde{\mathcal{V}}$  into two subsets $\mathcal V_1 \cup \mathcal V_2$ with $\mathcal V_1\cap  \mathcal V_2=\emptyset $, and call $\mathcal V_1$  the set of interior vertices, denoted by $\widetilde{\mathcal{V}}^o$,  and call $\mathcal V_2$ the set of boundary vertices, denoted by $\partial \widetilde{\mathcal{V}}$. 
Note, at this stage, the decomposition into $\mathcal V_1 \cup \mathcal V_2$ is arbitrary. 
When $v, w \in \widetilde{\mathcal{V}}$ are different end points of a same edge in $\widetilde{\mathcal{E}}$, we write $\{ v,w \} \in \widetilde{\mathcal{E}}$. We also  denote it by $v\stackrel{\widetilde{\mathcal{E}}}\sim w $.
For the sake of simplicity, we assume that $v\not\stackrel{\widetilde{\mathcal{E}}}\sim w $ if $v,w\in \partial \widetilde{\mathcal{V}}$.

The distance $d(v,w)$ on $\{ \widetilde{\mathcal{V}}^o\cup \partial \widetilde{\mathcal{V}}, \widetilde{\mathcal{E}} \} $ between two vertices $v,w\in \widetilde{\mathcal{V}}^o\cup \partial \widetilde{\mathcal{V}}$ is defined by the minimal number of edges in all paths from $v$ to $w$ in $\{ \widetilde{\mathcal{V}}^o\cup \partial \widetilde{\mathcal{V}} , \widetilde{\mathcal{E}} \}$. 
Note that $ d(v,w) \geq 2 $ when $v,w \in \partial \widetilde{ \mathcal{V}}$ under our assumption ($ v\not\stackrel{\widetilde{\mathcal{E}}}\sim w$ if $v,w \in \partial \widetilde{\mathcal{V}}$).

\begin{definition}
Given a subset $S\subset \widetilde{\mathcal{V}}^o$, we say  a vertex $v\in S$ is an \textit{extreme point} of $S$ with respect to $ \partial \widetilde{\mathcal{V}}$ if there exists a vertex $w \in \partial \widetilde{\mathcal{V}}$ such that $v$ is the unique nearest vertex in $S$ from $w$ with respect to the distance $d(\cdot ,\cdot )$ on $\{ \widetilde{\mathcal{V}}^o\cup \partial \widetilde{\mathcal{V}}, \widetilde{\mathcal{E}} \} $.
\label{S2_def_extremepoint}
\end{definition}
\begin{definition}
We say a finite graph $\{ \widetilde{\mathcal{V}}^o \cup \partial \widetilde{\mathcal{V}}, \widetilde{\mathcal{E}} \}$ satisfies the \textit{two-points condition} if the following is true: for any subset $S\subset \widetilde{\mathcal{V}}^o$ with $\sharp S \geq 2$, there exist at least two extreme points of $S$ with respect to $ \partial \widetilde{\mathcal{V}}$.
\label{S2_def_twopoints}
\end{definition}

As in \cite{BILL1,BILL2}, we impose the following condition.

\medskip

\noindent
{\bf (A-5)} For any $z \in \partial \widetilde{\mathcal{V}}$ and  $v,w \in \widetilde{\mathcal{V}}^o$ such that $v\stackrel{\widetilde{\mathcal{E}}}\sim z$ and $w\stackrel{\widetilde{\mathcal{E}}}\sim z$, we have $v\stackrel{\widetilde{\mathcal{E}}}\sim w $.

\medskip

Let us consider the discrete Schr\"{o}dinger operator $ - \widehat{ \Delta }_{\widetilde{\Gamma}} + \widehat{ V} $ on $\widetilde{\mathcal{V}}^o$ with the Neumann boundary condition on $ \partial \widetilde{\mathcal{V}}$, i.e.,
\begin{gather*}
\big((-\widehat{ \Delta}_{\widetilde{\Gamma}}+\widehat{ V} ) \widehat{ u} \big)(v)= \frac{1}{\deg (v)} \sum _{w\in \widetilde{\mathcal{V}}, w\stackrel{\widetilde{\mathcal{E}}}\sim v} \widehat{ u}(w) + \widehat{ V} (v)\widehat{ u}(v), \quad v\in \widetilde{\mathcal{V}}^o, \\ 
( \partial _{\nu}\widehat{ u} )(v)= \frac{1}{\deg (v)} \sum  _{w\in \widetilde{\mathcal{V}}^o, w\stackrel{\widetilde{\mathcal{E}}}\sim v} (\widehat{u} (w)- \widehat{ u} (v))=0 , \quad v\in \partial \widetilde{\mathcal{V}},
\end{gather*}
where $\deg (v)= \sharp \{ w\in \widetilde{\mathcal{V}}^o \cup \partial \widetilde{\mathcal{V}} \ ; \ w\stackrel{\widetilde{\mathcal{E}}}\sim v \} $ for every $v\in \widetilde{\mathcal{V}}^o\cup \partial \widetilde{\mathcal{V}}$.

See \cite[Lemma 2.4]{BILL2} for the following result.
\begin{lemma}
Let $\{ \widetilde{\mathcal{V}}^o \cup \partial \widetilde{\mathcal{V}}, \widetilde{\mathcal{E}} \}$ be a finite and connected graph as above.
Suppose that $\{ \widetilde{\mathcal{V}}^o \cup \partial \widetilde{\mathcal{V}}, \widetilde{\mathcal{E}} \}$ satisfies the two-points condition (Definition \ref{S2_def_twopoints}) and the assumption (A-5). 
Then, there is no non-trivial solution to $(-\widehat{\Delta}_{\widetilde{\Gamma}} + \widehat{ V} -\lambda )\widehat{ u}=0$ in $\widetilde{\mathcal{V}}^o$ satisfying both of the zero Dirichlet and the zero Neumann boundary conditions on $\partial \widetilde{\mathcal{V}}$. 
\label{S2_lem_ucp}
\end{lemma}

This UCP follows from a long-time UCP for an initial-boundary value problem of a discrete-time wave equation on $\{ \widetilde{\mathcal{V}}^o \cup \partial \widetilde{\mathcal{V}}, \widetilde{\mathcal{E}} \}$.
Now we introduce an assumption for UCP on a periodic lattice $\Gamma_0 = \{ \mathcal{V}_0 , \mathcal{E}_0 \} $.

\medskip
\noindent
{\bf (A-6)} For a sufficiently large constant $R > 0$, we put 
$$ 
\Omega _{R} = \cup_{j=1}^s \{ p_j + {\bf v} (x) \in \mathcal V_0 \ ; \ |x| \leq R \} \subset \mathcal V_0 .
$$
We assume that there exists a finite subsgraph $\widetilde{\Gamma}  = \{ \widetilde{\mathcal{V}} , \widetilde{\mathcal{E} } \} $ of $\Gamma_0$ for which the following statements hold true.
\begin{itemize}
\item $\widetilde{\mathcal{V}}$ is split into the interior vertices and the boundary vertices $\widetilde{\mathcal{V}}^o \cup \partial \widetilde{\mathcal{V}}$ with $ \widetilde{\mathcal{V}}^o \cap \partial \widetilde{\mathcal{V}} = \emptyset $. 
\item $\Omega_{R} \subset \widetilde{\mathcal{V}}^o $. 
\item Let $\widetilde{\mathcal{E}} \subset \mathcal{E}_0$ be defined by $\widetilde{\mathcal{E}} = \{ \{ v,w\} \in \mathcal{E}_0 \ ; \ v,w \in \widetilde{\mathcal{V}}^o \} \cup \{ \{ v,w\} \in \mathcal{E}_0 \ ; \ v\in \widetilde{\mathcal{V}}^o,w \in \partial \widetilde{\mathcal{V}} \ \text{or} \ v\in \partial \widetilde{\mathcal{V}},w \in  \widetilde{\mathcal{V}}^o  \} $. 
Then the finite subgraph $\widetilde{\Gamma} = \{ \widetilde{\mathcal{V}}^o \cup \partial \widetilde{\mathcal{V}}, \widetilde{\mathcal{E}} \}$ of $\Gamma_0$ satisfies the two-points condition (Definition \ref{S2_def_twopoints}) and the assumption (A-5).
\end{itemize}

\medskip

The assumption (A-6) guarantees UCP in the sense of Definition \ref{S2_def_ucp} as follows.

\begin{theorem}
If a finite subgraph $\widetilde{\Gamma}$ of a periodic lattice $\Gamma_0$ satisfies the assumption (A-6), UCP in the sense of Definition \ref{S2_def_ucp} holds true.
\label{S2_lem_ucpA5}
\end{theorem}

Proof.
Take a sufficiently large $R>0$.
Suppose that a function $\widehat{u}$ satisfies $ (\widehat{ H} -\lambda )\widehat{ u}=0$ on $\mathcal{V}_0$ and $\widehat{ u} (x)=0$ for $|x| > R $ for a constant $ \lambda $.
Then we can take a finite subgraph $\widetilde{\Gamma}$ satisfying (A-6).
In particular, it follows that $\widehat{ u}$ is a solution to $(-\widehat{ \Delta}_{\Gamma_0} +\widehat{ V}) \widehat{u} =0$ in $\widetilde{\mathcal{V}}^o$ with both zero Dirichlet and Neumann boundary conditions on $\partial \widetilde{\mathcal{V}}$.
Lemma \ref{S2_lem_ucp} implies $\widehat{u} =0$ in $\widetilde{\mathcal{V}}^o $ so that $\widehat{ u}$ identically vanishes on $\mathcal{V}_0$.
\qed

\medskip

We note that the assumption (A-6) can be generalized for perturbed lattices $\Gamma = \{ \mathcal{V},\mathcal{E} \} $ immediately.
Actually, it is sufficient to replace $\Omega_R$ by a sufficiently large finite subset $\Omega \subset \mathcal{V}$ and to suppose that there exists a finite subgraph $\widetilde{\Gamma}$ of $\Gamma$ satisfying the conditions of (A-6).
Then we obtain the nonexistence result of eigenvalues in the essential spectrum as follows.

\begin{theorem}
Under the assumption of Theorem \ref{Rellichtypetheorem} and (A-6), the operator $\widehat{ H}$ has no eigenvalues in $\sigma_{ess}(\widehat{ H}) \setminus \mathcal T_1$.
\end{theorem}

\subsection{Examples}
\subsubsection{Square shaped domain in the square lattice.}
Let $\widetilde{\mathcal{V}}^o = \Omega_R = \{ x\in {\bf Z}^d \ ; \ |x_j| \leq R  , \, j=1,\ldots,d\} $ for a large constant $R>0$ and $ \partial \widetilde{\mathcal{V}} = \{ x\in {\bf Z}^d \setminus \widetilde{\mathcal{V}}^o \ ; \ \exists y\in \widetilde{\mathcal{V}}^o \ \text{such that} \ x\sim y \} $.
Then $ (\widetilde{\mathcal{V}}^o \cup \partial \widetilde{\mathcal{V}} , \widetilde{\mathcal{E}} )$ is a square shaped domain in $\mathcal{V}_0 = {\bf Z}^d$.
In this case, we can prove UCP directly (see \cite{IsMo,IsMo2}).
Thus there is no eigenfunction of $\hat H$ on ${\bf Z}^d $ with a finite support.
Moreover, $ (\widetilde{\mathcal{V}}^o \cup \partial \widetilde{\mathcal{V}} , \widetilde{\mathcal{E}} )$ satisfies the two-points condition.
In fact, for any $S \subset \widetilde{\mathcal{V}}^o$ with $\sharp S \geq 2 $, every point $x\in S$ such that $x_1 = \max \{ y_1 \ ; \ y\in S \}$ or $\min \{ y_1 \ ; \ y\in S \} $ is an extreme point.  
 
\subsubsection{Parallelogram shaped domain in the hexagonal lattice}
\begin{figure}[t]
\begin{minipage}[t]{0.58\columnwidth}
\centering
\includegraphics[width=8cm, bb=0 0 1259 799]{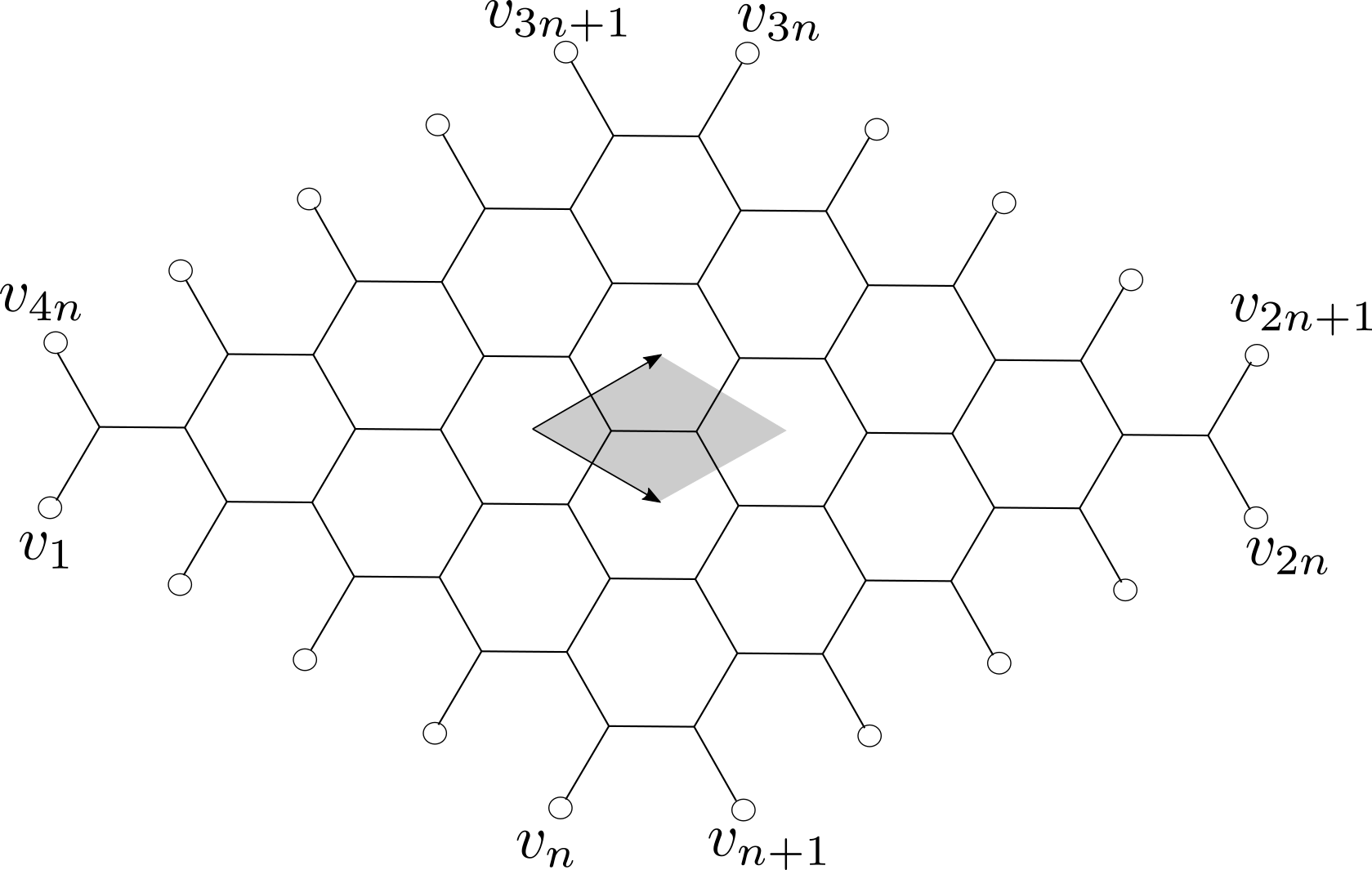}
%\caption{A subgraph $\{ \widetilde{\mathcal{V}}^o \cup \partial \widetilde{\mathcal{V}} , \widetilde{\mathcal{E}} \}$ of the hexagonal lattice. White vertices are the boundary vertices with the numbering $ \partial \widetilde{\mathcal{V}} = \{ v_1 , \ldots , v_{4n} \} $ for a positive integer $n$. 
%}
%\label{S2_fig_hexboundary}
\end{minipage}
\hspace*{0.02\columnwidth}
\begin{minipage}[t]{0.38\columnwidth}
\centering
\includegraphics[width=4cm, bb=0 0 589 926]{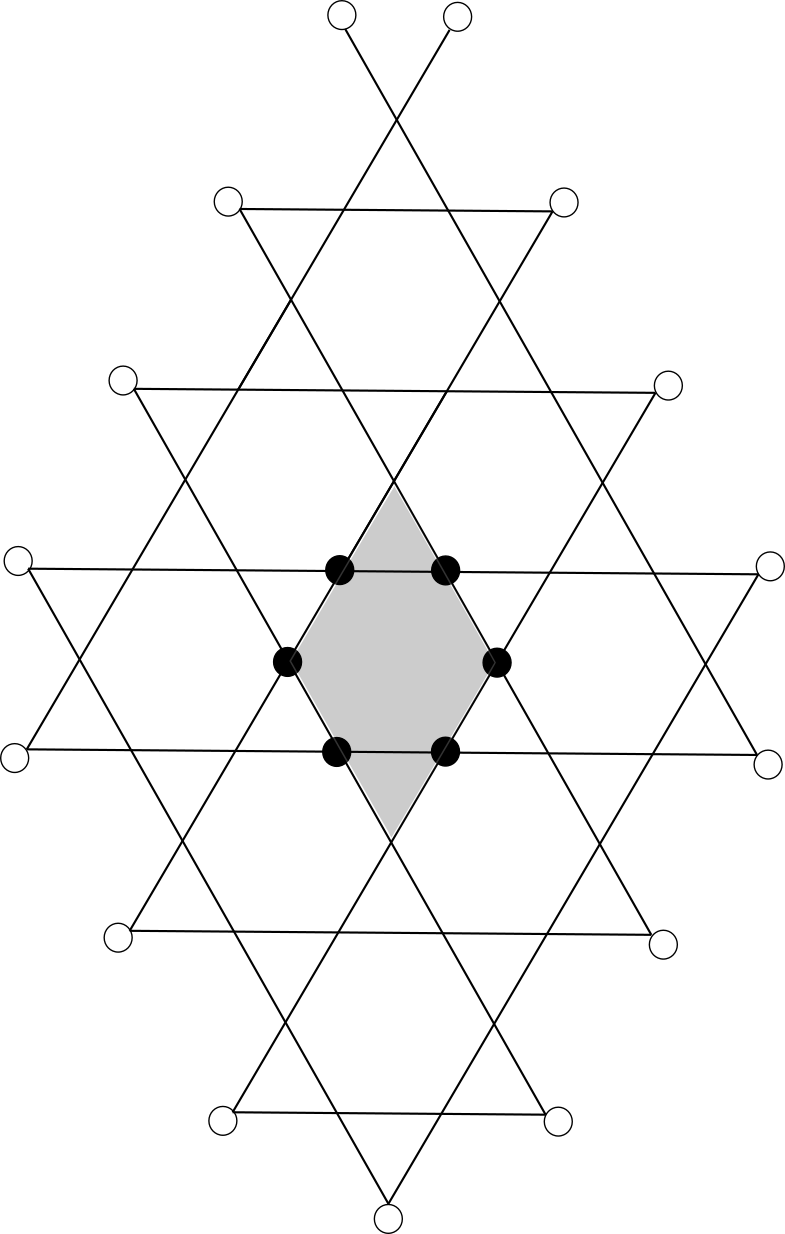}
%\caption{A subgraph $(\widetilde{\mathcal{V}} ^o \cup \partial \widetilde{\mathcal{V}} , \widetilde{\mathcal{E}} )$ of the kagome lattice. White vertices are the boundary vertices.
%This subgraph does not satisfy the two-points condition.
%In fact, the subset $S$ consisting of black vertices does not have extreme points.}
%\label{S2_fig_kagomeboundary}
\end{minipage}
\caption{Left : A subgraph $\{ \widetilde{\mathcal{V}}^o \cup \partial \widetilde{\mathcal{V}} , \widetilde{\mathcal{E}} \}$ of the hexagonal lattice. White vertices are the boundary vertices with the numbering $ \partial \widetilde{\mathcal{V}} = \{ v_1 , \ldots , v_{4n} \} $ for a positive integer $n$. / Right : A subgraph $(\widetilde{\mathcal{V}} ^o \cup \partial \widetilde{\mathcal{V}} , \widetilde{\mathcal{E}} )$ of the kagome lattice. White vertices are the boundary vertices.
This subgraph does not satisfy the two-points condition.
In fact, the subset $S$ consisting of black vertices does not have extreme points.}
\label{S2_fig_hexboundary_kagomeboundary}
\end{figure}
See Figure \ref{S2_fig_hexboundary_kagomeboundary}.
We consider the hexagonal lattice $\Gamma_0 = \left\{\mathcal{ V}_0 , \mathcal{E}_0 \right\}$.
Let $$ \widetilde{\mathcal{V}}^o =\Omega_R = \{ p_1 + {\bf v} (x), p_2 + {\bf v} (x) \in \mathcal{V}_0 \ ; \ \max (|x_1|,|x_2|)\leq R \} , $$ for a constant $R>0$ and $ \partial \widetilde{\mathcal{V}}= \{ v\in \mathcal{V}_0 \setminus \widetilde{\mathcal{V}} \ ; \ \exists w\in \widetilde{\mathcal{V}} \ \text{such that} \ v \sim w \} $.
In view of the numbering $ \partial \widetilde{\mathcal{V}} = \{ v_1 , \ldots , v_{4n} \}$ for a positive integer $n$, we take the shortest zigzag path $P_j$ from $v_j $ to $ v_{3n+1-j} $, or $Q_j$ from $v_{j+n} $ to $v_{4n+1-j} $ for $1\leq j \leq n $.
For any subset $S$ with $\sharp S=2$, evidently $S$ has two extreme points.
For any subset $S\subset \widetilde{\mathcal{V}}^o$ with $ \sharp S > 2$, we can take four zigzag lines $P_j$, $P_{j'}$, $Q_k$, and $Q_{k'} $ with $j' > j$ and $k' > k$ such that these zigzag lines enclose the subset $S$ and every zigzag line has at least one vertex in $S$.
If a vertex $v\in S $ is on $P _{j'}$, there exists a vertex $v_k \in \partial \widetilde{\mathcal{V}}$ with $k\in \{ n+1,\ldots,2n \}$ such that $v$ is the unique nearest vertex of $v_k$ in $S$.
For other zigzag lines $P_j$, $Q_k$ and $Q_{k'}$, the similar situation holds true.
Then $S$ has at least two extreme points, and $(\widetilde{\mathcal{V}}^o \cup \partial \widetilde{\mathcal{V}}, \widetilde{\mathcal{E}} )$ satisfies the two-points condition.

\subsubsection{Counter example (kagome lattice)}
\label{SubsubSCounterexample}
%\begin{figure}[h]
%\centering
%\includegraphics[width=4cm]{kagome_boundary.eps}
%\caption{A subgraph $(\widetilde{\mathcal{V}} ^o \cup \partial \widetilde{\mathcal{V}} , \widetilde{\mathcal{E}} )$ of the kagome lattice. White vertices are the boundary vertices.
%This subgraph does not satisfy the two-points condition.
%In fact, the subset $S$ consisting of black vertices does not have extreme points.}
%\label{S2_fig_kagomeboundary}
%\end{figure}
See Figure \ref{S2_fig_hexboundary_kagomeboundary}.
Here we consider the kagome lattice $\Gamma_0 = \left\{ \mathcal{V}_0 , \mathcal{E}_0 \right\}$.
Let $$ \widetilde{\mathcal{V}}^o =\Omega_R = \{ p_1 + {\bf v} (x), p_2 +{\bf v} (x),p_3 +{\bf v} (x) \in \mathcal{V}_0 \ ; \ \max (|x_1|,|x_2| )\leq R \} ,$$ for a constant $R>0$ and $ \partial \widetilde{\mathcal{V}} = \{ v\in \mathcal{V}_0 \setminus \widetilde{\mathcal{V}}^o \ ; \ \exists w\in \widetilde{\mathcal{V}}^o  \ \text{such that} \ v \sim w \} $.
As in Figure \ref{S2_fig_hexboundary_kagomeboundary}, we take a subset $S \subset \widetilde{\mathcal{V}}^o $ consisting of black vertices. 
$S$ does not have extreme points.
Then $(\widetilde{\mathcal{V}}^o \cup \partial \widetilde{\mathcal{V}} , \widetilde{\mathcal{E}} )$ does not satisfy the two-points condition.
In fact, as has been given in \cite[Section 5.3]{AIM}, we can construct a potential (cf. (\ref{S2_eq_Schrodinger})) for which there exists an eigenfunction having the subset $S$ as its support.
1

%%%%%%%%%%%%%%%%%%%%%%%%%%%%%%%%%%%%%
\appendix
\section{Connectivity of Fermi surfaces}

In this appendix, we show that the square lattice and the hexagonal lattice satisfy the assertion (2) in (A-4).

\subsection{Square lattice}
Let us recall some elementary properties of $\eta = \cos \zeta$ for $\zeta \in {\bf T}_{{\bf C}} $.
Letting
$$
\mathcal{D} _{\gamma} = \{ \zeta \in {\bf T} _{{\bf C}} \ ; \ |\mathrm{Im} \, \zeta |\leq \gamma \} , \quad \gamma \geq 0,
$$
and using 
$$
\mathrm{Re} \, \eta = \cos (\mathrm{Re} \, \zeta )\cosh (\mathrm{Im} \, \zeta), \quad \mathrm{Im} \, \eta = -\sin (\mathrm{Re} \, \zeta ) \sinh (\mathrm{Im} \, \zeta ),
$$
we see that the cosine function maps $ \mathcal{D}_{\gamma}$ to
\begin{gather*}
{\bf C} _{\gamma}^{\ast} := \cos (\mathcal{D}_{\gamma}) = \left\{ 
\begin{split}
\left\{ \eta \in {\bf C} \ ; \ \frac{(\mathrm{Re} \, \eta )^2}{(\cosh \gamma )^2} +\frac{(\mathrm{Im} \, \eta )^2}{(\sinh \gamma )^2} \leq 1 \right\} &, \quad \gamma >0, \\
[-1,1] &, \quad \gamma =0.
\end{split}
\right.
\end{gather*}
In fact, the case $\gamma > 0$ follows from the inequality
\begin{equation}
\begin{split}
\Big(\frac{{\rm Re}\, \eta}{\cosh \gamma}\Big)^2 + 
\Big(\frac{{\rm Im}\, \eta}{\sinh \gamma}\Big)^2 & = 
\cos^2x\Big(\frac{\cosh y}{\cosh \gamma}\Big)^2 +\sin^2x\Big(\frac{\sinh y}{\sinh \gamma}\Big)^2 \\
&\leq \cos^2 x + \sin^2 x = 1,
\end{split}
\nonumber
\end{equation}
where $\zeta = x + iy$ and $|y| \leq \gamma$. 
Therefore, ${\bf C} _{\gamma}^{\ast}$ is inside the ellipse for $\gamma >0$.

\begin{prop}
Let $\Gamma_0 = \{ \mathcal{V}_0 , \mathcal{E}_0 \} $ be the $d$-dimensional square lattice with $d\geq 2 $.
For any constant $a>0$, $M_{\lambda,reg}^{{\bf C}} $ with $\lambda \in \sigma (H_0) \setminus \{ -1,1\} $ satisfies the assertion (2) in the assumption (A-4).
\label{S3_prop_squarelattice}
\end{prop}

Proof.
For $\lambda \in (-1,1)$, the equation $p(z,\lambda )=0$ on ${\bf T}^d _{{\bf C}} $ is equivalent to 
\begin{equation}
(\cos z_1 +\lambda ) + \cdots + (\cos z_d +\lambda ) =0  .
\label{S3_eq_squarelattice001}
\end{equation} 
Letting $w_j = \cos z_j$, the equation (\ref{S3_eq_squarelattice001}) is rewritten as 
\begin{gather}
(\mathrm{Re} \, w_1 +\lambda ) + \cdots + (\mathrm{Re} \, w_d +\lambda )=0 , \label{S3_eq_squarelattice002} \\
\mathrm{Im} \, w_1 + \cdots + \mathrm{Im} \, w_d =0 . \label{S3_eq_squarelattice003}
\end{gather}
We take a point $z^{(0)} \in M_{\lambda,reg}^{{\bf C}} \cap \Omega_a $, and construct a continuous curve $c (t)=(c_1 (t),\ldots,c_d (t))$ in $M_{\lambda,reg}^{{\bf C}} \cap \Omega_a$ for $t\in [0,1]$ such that $c (0)=z^{(0)} $ and $c (1)\in M_{\lambda,reg}^{{\bf C}} \cap {\bf T}^d $. 

\smallskip
\noindent
(i) We first translate $z^{(0)}$ within $M_{\lambda,reg}^{{\bf C}} \cap \Omega_a$ to the real submanifold. Note that $M_{\lambda,sng}^{{\bf C}} \subset \{ z\in {\bf T}^d_{{\bf C}} \ ; \ z_j = 0 \ \text{modulo} \ \pi , \ j=1,\ldots ,d \} $. As $z^{(0)} \in \Omega_a$, 
there exist constants $\gamma_1 , \ldots,\gamma_d > 0$ such that $z_j^{(0)}\in \mathcal{D} _{\gamma_j} $ and $\gamma_1^2 +\cdots+\gamma_d^2 <a^2 $.
We put 
\begin{equation}
w_j^{(0)} = \cos z_j^{(0)} \in {\bf C} _{\gamma_j}^{\ast} , \quad j=1,\ldots,d.
\label{definewj(0)}
\end{equation}
We rearrange $w_j^{(0)}$'s so that
\begin{equation}
0\leq |\mathrm{Im} \, w_1^{(0)} |\leq |\mathrm{Im} \, w_2^{(0)} |\leq \cdots \leq |\mathrm{Im} \, w_d^{(0)} |.
\label{S2_eq_hermiconnectivity_Im01}
\end{equation}
In the following argument, for a curve $c_j(t)$ in ${\mathcal D}_{\gamma_j}$ for $\gamma _j > 0$, we put 
\begin{equation}
c_j^* (t)= \cos c_j (t) \in {\bf C} _{\gamma_j}^{\ast} .
\label{definecjast(t)}
\end{equation}
Without loss of generality, we assume that $\mathrm{Im} \, w_1^{(0)} \geq 0 $. 
Then, by (\ref{S3_eq_squarelattice003}) and (\ref{S2_eq_hermiconnectivity_Im01}), there exists $k\in \{ 2,\ldots ,d\} $ such that $\mathrm{Im} \, w_k^{(0)} \leq 0$ and $0 \leq \mathrm{Im} \, w_1^{(0)} \leq |\mathrm{Im} \, w_k^{(0)} |$. 
Let $t_j = j/5, j = 0, 1, \dots,5$.
By virtue of (\ref{definewj(0)}) and the fact that ${\bf C} _{\gamma_1}^{\ast}$ is inside the ellipse,
we can take a curve $c_1^{\ast}(t)$ in ${\bf C}_{\gamma_1}^{\ast}$ for $t\in [0,t_1]$ such that $c_1(0) = z_1^{(0)}$, $\mathrm{Im} \, c_1^* (t) \downarrow 0$ as $t\to t_1$ and $\mathrm{Re} \, c_1^* (t)=\mathrm{Re} \, w_1 ^{(0)} $ on $[0,t_1]$.
For $j\in \{ 2,\ldots,d\} \setminus \{ k \}$, we put $c_j^* (t)=w_j^{(0)}$ for $t\in [0,t_1] $, and construct $c_k^{\ast}(t)$ by using (\ref{S3_eq_squarelattice002}) and 
(\ref{S3_eq_squarelattice003}). The curve $c(t)$ on $[0,t_1]$ is then defined by (\ref{definecjast(t)}). Taking account of
$$
\mathrm{Im}\, c_1^{\ast}(t) + \mathrm{Im}\,c_k^{\ast}(t) = 
- \sum_{j\neq 1,k}\mathrm{Im}\, w_j^{(0)} = \mathrm{Im}\, w_1^{(0)} + \mathrm{Im}\, w_k^{(0)} \leq 0
$$
and the fact that $\mathrm{Im} \, c_1^{\ast}(t)$ is non-negative and monotone decreasing, we see that the curve $c_k^* (t)$ satisfies $\mathrm{Im} \, c_k^* (t)\leq 0$ and $|\mathrm{Im} \, c_k^* (t)|\leq |\mathrm{Im} \, w_k^{(0)} |$ on $[0,t_1]$. 
To the curve thus constrcuted, we apply the same procedure. Then, we can take a new curve $c(t)\in M_{\lambda,reg}^{{\bf C}} \cap \Omega_a$ on $[0,t_2]$ such that $\mathrm{Re} \, c_j^* (t_2)=\mathrm{Re} \, w_j^{(0)}$ and $\mathrm{Im} \, c_j^* (t_2 )=0$ for $j=1,\ldots,d$.

\smallskip
\noindent
(ii) We next  translate in the real manifold.
Without loss of generality, we assume that 
\begin{equation}
0\leq |\mathrm{Re} \, w_1^{(0)} +\lambda |\leq |\mathrm{Re} \, w_2^{(0)} +\lambda|\leq \cdots \leq |\mathrm{Re} \, w_d^{(0)} +\lambda|.
\label{S2_eq_hermiconnectivity01}
\end{equation}
If $\pm (\mathrm{Re} \, w_1 ^{(0)} +\lambda )\geq 0$, there exists $k\in \{ 2,\ldots,d \}$ such that $\pm (\mathrm{Re} \, w_k^{(0)} +\lambda )\leq 0$ and $|\mathrm{Re} \, w_1^{(0)}+\lambda|\leq | \mathrm{Re} \, w_k^{(0)}+\lambda|$ by (\ref{S2_eq_hermiconnectivity01}) and the equation (\ref{S3_eq_squarelattice002}). 
We take a curve $c_1 (t)$ such that $\pm (\mathrm{Re} \, c_1^* (t)+\lambda ) \downarrow 0 $ as $t\to t_3$ and $\mathrm{Im} \, c_1^* (t)=0$ for $t\in [t_2 ,t_3 ]$.
For $j\in \{ 2,\ldots,d\} \setminus \{ k \}$, we put $c_j^* (t)=c_j^* (t_2)$ for $t\in [t_2,t_3] $.
By the same argument as in (i), it follows that the curve $c_k^* (t)$ for $t\in [t_2,t_3]$ satisfies $\pm (\mathrm{Re} \, c_k^* (t_3)+\lambda )\leq 0$ and $|\mathrm{Re} \, c_k^* (t)+\lambda |\leq |\mathrm{Re} \, w_k^{(0)} +\lambda |$.
Note that $c_j^* (t) \in {\bf C}_{\gamma_j}^{\ast} $, hence $c_j (t)\in \mathcal{D}_{\gamma_j}$ for $t\in [t_2 ,t_3  ]$, $j=1,\ldots,d$.
Therefore, we obtain $c(t)\in M_{\lambda,reg}^{{\bf C}} \cap \Omega_a$ on $ [t_2 , t_3 ]$.

Repeating similar procedures, we can take a curve $c(t)\in M_{\lambda,reg}^{{\bf C}} \cap \Omega_a$ for $t\in [t_3 , t_4 ] $ satisfying
\begin{gather}
\mathrm{Im} \, c_j ^* (t )=0 , \quad t\in [t_3,t_4 ], \quad j=1,\ldots,d, \label{S2_eq_hermiconnectivity03}\\
\mathrm{Re} \, c_1^* (t_4 )+\lambda = \cdots = \mathrm{Re} \, c_{d-2}^* (t_4 )+\lambda =0, \label{S2_eq_hermiconnectivity02} \\
|\mathrm{Re} \, c_{d-1}^* (t_4)+\lambda |\leq |\mathrm{Re} \, c_d^* (t_4)+\lambda |. 
\label{S2_eq_hermiconnectivity03}
\end{gather}
Due to the equations (\ref{S3_eq_squarelattice002}) and (\ref{S2_eq_hermiconnectivity02}), if $\pm (\mathrm{Re} \, c_{d-1}^* (t_4)+\lambda )>0$, we have $\pm (\mathrm{Re} \, c_d^* (t_4)+\lambda )<0$.
We take a curve $c_{d-1}^* (t)$ for $t\in [t_4 ,1]$ such that $\pm (\mathrm{Re} \, c_{d-1}^* (t)+\lambda ) \downarrow 0$ as $t\to t_5$ for $t\in [t_4 , 1]$ in the same way for $c_1^* (t)$ with $t\in [t_2 , t_3 ]$.
For $j\in \{ 1,\ldots,d-2 \}$, we put $c_j^* (t)= c_j^* (t_4)=-\lambda $ for $t\in [t_4 ,1 ]$.
Due to the equation (\ref{S3_eq_squarelattice002}) and (\ref{S2_eq_hermiconnectivity02}), we see $\mathrm{Re} \, c_d^* (t)+\lambda \to 0$ as $t\to 1$.
By the construction, $c (t)\in M_{\lambda,reg}^{{\bf C}} \cap \Omega_a$ on $[0,1 ]$.

\smallskip
\noindent
(iii) Finally, we note that $\mathrm{Re} \, c_1^* (1)=\cdots=\mathrm{Re} \, c_d ^* (1)=-\lambda \in (-1,1)$ and $\mathrm{Im} \, c_1^* (1)=\cdots =\mathrm{Im} \, c_d^* (1)=0$.
We have thus constructed a continuous curve $c(t)$ in $M_{\lambda,reg}^{{\bf C}} \cap \Omega_a$ for $t\in [0,1]$ such that $c(0)=z^{(0)}$ and $c(1)\in M_{\lambda,reg}^{{\bf C}} \cap {\bf T}^d$.
\qed

\subsection{Hexagonal lattice}

We consider the hexagonal lattice $\Gamma_0 = \{ \mathcal{V}_0 , \mathcal{E}_0 \} $.
Let $\lambda \in (-1,1)\setminus \{ 0 \}$.
In this case, the equation $p(z,\lambda )=0$ on ${\bf T}^2_{{\bf C}} $ is equivalent to 
$$
\cos z_1 + \cos z_2 + \cos (z_1 -z_2 )= \rho , \quad \rho = \frac{3}{2} (3\lambda ^2 -1)\in (-3/2 ,3) .
$$
Letting $\zeta_1 = (z_1 +z_2 )/2$ and $\zeta_2 = (z_1 -z_2 )/2$, the above equation is equivalent to
\begin{equation}
\cos \zeta_2 (\cos \zeta_1 +\cos \zeta_2 )= \mu , \quad \mu = \frac{1}{2} (\rho +1)\in (-1/4,2) .
\label{Ap_eq_hexagonal_cha}
\end{equation}
Let $w_j = \cos z_j$ and $\eta_j = \cos \zeta_j$ for $j=1,2$.
For $z\in \Omega_a$,
in view of $|\mathrm{Im} \, z_1 |^2 + |\mathrm{Im} \, z_2 |^2 = 2(|\mathrm{Im} \, \zeta_1 |^2 +|\mathrm{Im} \, \zeta_2 |^2 )$, we have $\zeta \in \Omega_{a/\sqrt{2}} $.
Then we see that $w_j \in {\bf C} _{\gamma_j}^{\ast}$, $\eta_j \in {\bf C} _{\gamma'_j}^{\ast}$, where $\gamma_1^2 + \gamma_2^2 <a^2$, $(\gamma'_1)^2 +(\gamma'_2 )^2 <a^2 /2$.

\begin{prop}
Let $\Gamma_0 = \{ \mathcal{V}_0 , \mathcal{E}_0 \} $ be the hexagonal lattice.
For any constant $a>0$, $M_{\lambda,reg}^{{\bf C}} $ with $\lambda \in \sigma (H_0) \setminus \{ -1,0,1\} $ satisfies the assertion (2) in the assumption (A-4).
\label{S3_prop_squarelattice}
\end{prop}

Proof. 
Take a point $z\in M_{\lambda,reg}^{{\bf C}} \cap \Omega_a$, and let $\zeta$, $\eta$ be as above. Then,  we have by (\ref{Ap_eq_hexagonal_cha}), $\mu \in (-1/4,2)$ and
\begin{equation}
\eta_2 (\eta_1 +\eta_2 )=\mu .
\label{Ap_eq_hexisospectraleq}
\end{equation}
We consider the cases $\mu =0 $, $\mu \in (0,2)$ and $\mu \in (-1/4,0 ) $, separately.

\smallskip
\noindent
(i) If $\mu =0$, then $\eta_2 =0 $ or $\eta_1 + \eta_2 =0 $, and the lemma is obvious.

\smallskip
\noindent
(ii) We consider the case that $\mu \in (0,2)$.
The equation (\ref{Ap_eq_hexisospectraleq}) implies $\eta_2 \neq 0$, and
\begin{equation}
\eta_1 = \frac{\mu}{\eta_2} -\eta_2 .
\label{Ap_eq_hexisospectraleq02}
\end{equation}
Representing $\eta_2 = re^{i\theta}$, with $r>0$ and $\theta \in [0,2\pi )$, we have from (\ref{Ap_eq_hexisospectraleq02}) that 
\begin{gather}
\eta_1 = \left\{
\begin{split}
\left( \frac{\mu}{r} -r  \right) \cos (-\theta) +i \left( \frac{\mu}{r} +r  \right) \sin (-\theta ) &, \quad r\not= \sqrt{\mu} , \\
2i \sqrt{\mu} \sin (-\theta ) &, \quad r=\sqrt{\mu} .
\end{split}
\right.
\label{Ap_eq_hexisospectraleq03}
\end{gather}
For a fixed $r\not= \sqrt{\mu}$, 
$\eta_1$ given by (\ref{Ap_eq_hexisospectraleq03})  is on the ellipse
\begin{equation}
\frac{(\mathrm{Re} \, \eta_1 )^2 }{(\mu r^{-1} -r)^2 } + \frac{(\mathrm{Im} \, \eta_1 )^2 }{(\mu r^{-1} +r)^2 } =1
\label{Ap_eq_hexisospectraleq04}
\end{equation}
with  focuse $\pm 2i\sqrt{\mu} $. 
Here, note that
$$
4\mu \leq \left( \frac{\mu }{r} + r \right)^2 = \left( \frac{\mu}{r} -r \right)^2 +4\mu <\left( \frac{\mu}{r} -r \right)^2 +8.
$$
We take $ \theta (t) \in C([0,1/2 ];{\bf R})$ such that $\theta (0)=\theta$, $|\sin \theta (t)| \downarrow 0$  as $t\to 1/2$, and $c_2^* (t):=  re^{i \theta (t)} \in {\bf C} _{\gamma'_2 }^{\ast} $.
As $\eta _1$ and $ \eta_2$ are inside the ellipses ${\bf C}^*_{\gamma'_1} $ and ${\bf C}^* _{\gamma'_2}$ respectively, we have $(\mu r^{-1} -r )^2 \leq (\cosh \gamma'_1 )^2$.
We take $c_1^* (t)\in {\bf C} _{\gamma'_1}^{\ast} $ for $t \in [0,1/2] $, where
\begin{gather*}
c_1^* (t)= \left\{
\begin{split}
\left( \frac{\mu}{r} -r  \right) \cos (-\theta (t)) +i \left( \frac{\mu}{r} +r  \right) \sin (-\theta (t) ) &, \quad r\not= \sqrt{\mu} , \\
2i \sqrt{\mu} \sin (-\theta (t) ) &, \quad r=\sqrt{\mu} .
\end{split}
\right.
\end{gather*}
Note that $c_2^* (1/2) = \pm r$ and $c_1^* (1/2)= \mu (\pm r)^{-1} -(\pm r)$, respectively.
Take $r(t) \in C([1/2,1];{\bf R})$ such that $r(1/2)=\pm r$, $r(t)\not= 0$, $|r(t)|\leq \cosh \gamma'_1 $, and 
$$
r(1)\in ((-1+\sqrt{1+4\mu})/2 ,(1+\sqrt{1+4\mu})/2)\cap (-1,1).
$$
Letting $c_1^* (t)= \mu r(t)^{-1 } -r(t)$ and $c_2^*(t) = r(t)$ for $t\in [1/2 ,1]$, we see $c_1^* (1),c_2 ^* (1)\in (-1,1) $. 
Taking the corresponding continuous curve $c(t) \in M_{\lambda,reg}^{{\bf C}} \cap \Omega_a$ for $t\in [0,1]$ with 
$$
c_1^* (t)= \cos \frac{c_1 (t)+c_2(t)}{2}, \quad c_2^* (t)= \cos \frac{c_1 (t)-c_2 (t)}{2},
$$
we obtain the lemma for $\mu \in (0,2)$.

\smallskip
\noindent
(iii) We consider the case of $\mu \in (-1/4,0 )$.
We again put $\eta_2 = re^{i\theta} $ and take a similar curve $c_2^* (t)= r\sin \theta (t)$ for $t\in [0,1/2]$.
It follows from (\ref{Ap_eq_hexisospectraleq02}) that
\begin{gather}
\eta_1 = \left\{
\begin{split}
\left( \frac{\mu}{r} -r  \right) \cos (-\theta) +i \left( \frac{\mu}{r} +r  \right) \sin (-\theta ) &, \quad r\not= \sqrt{-\mu} , \\
-2 \sqrt{-\mu} \cos (-\theta ) &, \quad r=\sqrt{-\mu} ,
\end{split}
\right.
\label{Ap_eq_hexisospectraleq05}
\end{gather}
with 
$$
\left( \frac{\mu }{r} - r \right)^2 = \left( \frac{\mu }{r} + r \right)^2 +4(-\mu )<\left( \frac{\mu }{r} + r \right)^2 +1.
$$
Then $\eta_1$ given by (\ref{Ap_eq_hexisospectraleq05}) for $r\not= \sqrt{-\mu}$ is on the ellipse (\ref{Ap_eq_hexisospectraleq04}) with  focuse $\pm 2\sqrt{-\mu}$.
Here, let us remark that $(\mu r^{-1} -r)^2 \leq (\cosh \gamma'_1 )^2 $. In fact, 
if $(\mu r^{-1} -r)^2 > (\cosh \gamma'_1 )^2 $, we have $(\mu r^{-1} +r )^2 > (\sinh \gamma'_1 )^2 $.
Then, contrary to our assumption, the ellipse (\ref{Ap_eq_hexisospectraleq04}) lies outside ${\bf C} _{\gamma'_1} ^*$.
As $z\in M_{\lambda,reg}^{{\bf C}} \cap \Omega_a$, the ellipse (\ref{Ap_eq_hexisospectraleq04}) has an intersection with ${\bf C} _{\gamma'_1} ^*$.
We thus obtain $c_1^* (t)\in {\bf C} _{\gamma'_1}^*$ for $t\in [0,1/2]$ where 
\begin{gather*}
c_1^* (t)= \left\{
\begin{split}
\left( \frac{\mu}{r} -r  \right) \cos (-\theta (t)) +i \left( \frac{\mu}{r} +r  \right) \sin (-\theta (t) ) &, \quad r\not= \sqrt{-\mu} , \\
-2 \sqrt{-\mu} \cos (-\theta (t) ) &, \quad r=\sqrt{-\mu} .
\end{split}
\right.
\end{gather*}
The remaining part of the proof is parallel to the case of $\mu \in (0,2)$.
\qed

%%%%%%%%%%%%%%%%%%%%%%%%%%%%%%%%%%%%%

\medskip

{\bf Acknowledgements.}
K. Ando is supported by JSPS grant-in-aid for Scientific Research (C) No. 20K03655. 
H. Isozaki is supported by JSPS grant-in-aid for Scientific Research (C) No. 24K06768. 
H. Morioka is supported by JSPS grant-in-aid for scientific research (C) No. 24K06761. 
They are indebted for the supports.
The authors greatly appreciate the helpful comments by the reviewer for the first version of this manuscript.
The description of this paper has been significantly improved by these comments.

\end{document}